\documentclass[leqno,preprint]{imsart}
\usepackage[dvips]{graphics}
\DeclareGraphicsExtensions{.eps.gz,.eps,.epsi.gz,.epsi,.ps,.ps.gz}
\usepackage{epsfig}
\usepackage{color}
\graphicspath{{fig/}} 

\usepackage{graphicx}
\usepackage{wrapfig}
\usepackage{lscape}
\usepackage{rotating}
\usepackage{epstopdf}

\usepackage{amsmath}
\usepackage{mathtools}
\usepackage{amsthm}

\usepackage{hyperref}
\usepackage{cleveref}
\hypersetup{breaklinks=true,
    linkcolor=blue,
    citecolor=blue,
    colorlinks=true,
pdfborder={0 0 0}}

\usepackage{latexsym}
\usepackage{graphicx}
\usepackage{amsfonts}
\usepackage{amssymb}
\usepackage{mathrsfs}
\usepackage[lofdepth,lotdepth]{subfig}
\usepackage{multirow}
\usepackage{booktabs}
\usepackage{lscape}

\usepackage{thmtools}
\usepackage{thm-restate}

\numberwithin{equation}{section}

\def\argmax{\mathop{\rm arg\, max}}
\def\argmin{\mathop{\rm arg\, min}}

\newcommand{\bel}{\begin{eqnarray}\label}
\newcommand{\eel}{\end{eqnarray}}
\newcommand{\bes}{\begin{eqnarray*}}
\newcommand{\ees}{\end{eqnarray*}}
\newcommand{\bei}{\begin{itemize}}
\newcommand{\eei}{\end{itemize}}
\newcommand{\beiftnt}{\begin{itemize}\footnotesize}

\def\benu{\begin{enumerate}}
\def\eenu{\end{enumerate}}

\def\argmax{\mathop{\rm arg\, max}}
\def\argmin{\mathop{\rm arg\, min}}
\def\real{{\mathbb{R}}}
\def\R{{\real}}

\def\E{{\mathbb{E}}}

\def\P{{\mathbb{P}}}

\def\complex{\mathop{{\rm I}\kern-.58em\hbox{\rm C}}\nolimits}
\def\pa{\partial}

\def\rank{\hbox{\rm rank}}

\def\sgn{\hbox{\rm sgn}}

\DeclareMathOperator{\trace}{trace}

\def\Var{\hbox{\rm Var}}

\DeclareMathOperator{\supp}{supp}

\def\mathbold{\boldsymbol} 

\def\ba{\mathbold{a}}

\def\bA{\mathbold{A}}

\def\bb{\mathbold{b}}

\def\bB{\mathbold{B}}

\def\bg{\mathbold{g}}

\def\bh{\mathbold{h}}

\def\bI{\mathbold{I}}

\def\bP{\mathbold{P}}

\def\bQ{\mathbold{Q}}

\def\Rhat{\widehat{R}}

\def\bs{\mathbold{s}}

\def\Shat{{\widehat{S}}}

\def\be{\mathbold{e}}

\def\bu{\mathbold{u}}

\def\bv{\mathbold{v}}

\def\bw{\mathbold{w}}

\def\bx{\mathbold{x}}

\def\bX{\mathbold{X}}

\def\by{\mathbold{y}}

\def\bY{\mathbold{Y}}

\def\bz{\mathbold{z}}

\def\bZ{\mathbold{Z}}

\def\bbeta{\mathbold{\beta}}

\def\hbbeta{\hat{\bbeta}}

\def\bbetabar{{\overline\bbeta}}

\def\ep{\varepsilon}\def\eps{\epsilon}
\def\bep{ {\mathbold{\ep} }}

\def\btheta{\mathbold{\theta}}

\def\hbtheta{{\widehat{\btheta}}}

\def\lam{\lambda}

\def\bmu{\mathbold{\mu}}

\def\hbmu{{\widehat{\bmu}}}\def\tbmu{{\widetilde{\bmu}}}

\def\bSigma{\mathbold{\Sigma}}

\def\bSigmabar{{\overline\bSigma}}

\declaretheorem[name=Theorem,numberwithin=section]{theorem}
\declaretheorem[name=Proposition,sibling=theorem]{proposition}
\declaretheorem[name=Lemma,sibling=theorem]{lemma}
\declaretheorem[name=Corollary,sibling=theorem]{corollary}
\declaretheorem[name=Assumption,numberwithin=section]{assumption}

\declaretheorem[name=Remark,style=remark,numberwithin=section]{remark}

\def\argmax{\mathop{\rm arg\, max}}
\def\argmin{\mathop{\rm arg\, min}}

\def\SURE{\widehat{\mathsf{{\scriptstyle SURE}}}{}}

\def\supp{\hbox{\rm supp}}

\def\RE{{\rm{RE}}}

\DeclareMathOperator{\dv}{div}

\def\hbmu{{\widehat{\bmu}}{}}
\def\df{{\hat{\mathsf{df}}}{}}

\usepackage[dvipsnames]{xcolor}
\usepackage{soul}

\usepackage[textsize=scriptsize,textwidth=1.7in]{todonotes}
\usepackage{marginnote}

\usepackage[color]{changebar} 
\cbcolor{magenta}
\def\tSURE{{{\textsc{\tiny {sure}}}} }
\def\tEN{{{\textsc{\tiny {EN}}}} }
\def\tLasso{{{\textsc{\tiny {lasso}}}} }
\def\tdiff{{{\text{\tiny (diff)}}} }

\def\lasso{\hbbeta{}_\tLasso^{(\lam)}}
\newcommand\lassoWith[1]{\hbbeta{}_\tLasso^{(#1)}}

\begin{document}

\runtitle{Second Order Stein} 
\title{ Second Order Stein: SURE for SURE and Other Applications in
High-Dimensional Inference}
\date{\today}
\begin{aug}
  \author{\fnms{Pierre C.}  \snm{Bellec}\thanksref{t1}\ead[label=e1]{pierre.bellec@rutgers.edu}}
  \and
  \author{\fnms{Cun-Hui} \snm{Zhang}\thanksref{t2}\ead[label=e2]{czhang@stat.rutgers.edu}}

  \thankstext{t1}{
    Research partially supported by the NSF Grant DMS-1811976 
    and NSF CAREER award DMS-1945428.
  }
  
  \thankstext{t2}{ Research partially supported by the NSF Grants DMS-1513378,
  IIS-1407939, DMS-1721495 and IIS-1741390.  }

  \runauthor{Bellec and Zhang}

  \affiliation{Rutgers University}

  \address{Department of Statistics, Hill Center, Busch Campus, \\
      Rutgers University, Piscataway, NJ 08854, USA. \\
  \printead{e1,e2}}

\end{aug}

\begin{abstract}
    Stein's formula states that a random variable of the form
    $\bz^\top f(\bz) - \dv f(\bz)$ is mean-zero for all functions $f$ with integrable gradient.
    Here, $\dv f$ is the divergence of the function $f$ and $\bz$ is a standard normal vector. 
    This paper aims to propose a Second Order Stein formula to characterize the 
    variance of such random variables for all functions $f(\bz)$ with square integrable gradient,  
    and to demonstrate the usefulness of this Second Order Stein formula in various applications. 

    In the Gaussian sequence model, a remarkable consequence of Stein's formula is 
    Stein's Unbiased Risk Estimate 
    (SURE), an unbiased estimate of the mean squared risk for almost any given estimator $\hbmu$ 
    of the unknown mean vector.
    A first application of the Second Order Stein formula 
    is an Unbiased Risk Estimate for SURE itself (SURE for SURE):
    an unbiased estimate {providing} information about the squared
    distance between SURE and the squared estimation error of $\hbmu$.
    SURE for SURE 
    has a simple form as a function of the data and 
    is applicable to all $\hbmu$ with square integrable gradient, 
    for example the Lasso and the Elastic Net.

    In addition to SURE for SURE, the following statistical applications
    are developed: (1) Upper bounds on the risk of SURE when the estimation
    target is the mean squared error;
    (2) Confidence regions based on SURE and using the Second Order Stein formula; 
    (3) Oracle inequalities satisfied by SURE-tuned estimates under
    a mild Lipschtiz assumption;
    (4) An upper bound on the variance of the size of the model selected by the Lasso,
    and more generally an upper bound on the variance of the empirical degrees-of-freedom
    of convex penalized estimators;
    (5) Explicit expressions of SURE for SURE for the Lasso and the Elastic-Net;
    (6)
    In the linear model, a general semi-parametric scheme to de-bias a 
    differentiable initial estimator for the statistical inference of a
    low-dimensional projection of the unknown regression coefficient vector,
    with a characterization of the variance after de-biasing;
    and (7) An accuracy analysis of a Gaussian Monte Carlo scheme to approximate
    the divergence of functions $f:\R^n\to\R^n$.

   

\end{abstract}

\maketitle



\section{Introduction}

    \subsection{Stein's formula}
    The univariate version of Stein's formula reads $\E[Z g(Z)]=\E[g'(Z)]$ where $Z\sim N(0,1)$
    and $g$ is absolutely continuous with finite $\E[|g'(Z)|]$\footnote{
    Existence of $\E|Zg(Z)|<+\infty$ follows from $\E|g'(Z)|<+\infty$.
        Indeed by Fubini's theorem
        $\E |\max(Z,0)\{g(Z)-g(0)\}|
        \le \int_{0}^\infty |g'(u)|\E[Z I_{\{Z\ge u\}}] du
        = \E[I_{Z>0}|g'(Z)|]$
        and similarly for $Z<0$. 
    }.
In \cite{stein1972bound}, Stein observed that the set of equations $\E[Z g(Z) - g'(Z)]=0$
for all absolutely continuous $g$ with integrable gradient characterize the standard normal
distribution. Stein went on to show that if $\E[Wg(W) -g'(W)]$ is small for a large class of functions $g$,
then $W$ is close in distribution to $N(0,1)$, see for instance
\cite[Proposition 1.1]{chatterjee2014short} for a precise statement.
Since then, this powerful technique has been used to obtain central limit
theorems under dependence;
we refer the reader to the book \cite{chen2010normal}
for a recent survey on Stein's formula and its application to normal approximation. 

The multivariate version of Stein's formula \cite{stein1981estimation}
can be described as follows.
Let $\bz=(z_1,...,z_n)^{\top} {\sim N({\bf 0},\bI_n)}$ be a standard normal random vector.
Let $f_1,...,f_n$ be functions $\R^n\to\R$
and denote by $f$ the column vector in $\R^n$ 
with $i$-th component equal to $f_i$.
If each $f_i$ is weakly differentiable with respect to $z_i$ and
    $\E[|f_i(\bz)|+|(\partial/\partial z_i)f_i(\bz)|]<+\infty$ 
    for each $i\in[n]$,
then
\begin{equation}
\E[\bz^\top f(\bz)] = \E[\dv f(\bz)]
\label{stein-first-order-formula}
\end{equation}
holds, where {${\dv } f = \sum_{i=1}^n (\partial/\partial x_i)f_i$ is the divergence of $f$.} 
The central object of the paper is the random variable $\bz^\top f(\bz) - \dv f(\bz)$,
which is mean-zero in virtue of \eqref{stein-first-order-formula}. We provide
in the next section an exact identity for the second moment of $\bz^\top f(\bz) - \dv f(\bz)$
that we termed \emph{Second Order Stein formula}.

\subsection{Related work}
Iterated Stein formulae have appeared in several works. If $g:\R^p\to\R$ is smooth with compact support,
iterating the univariate Stein's formula directly yields
$\E[z_iz_j g(\bz)]=\E[(\partial/\partial x_j)(\partial/\partial x_i) g(\bz)]$ for $i\ne j$
as well as $\E[(z_i^2 -1)g(\bz)] = \E[(\partial^2/\partial x_i^2)g(\bz)]$.
This may be succinctly written as
\begin{equation}
    \label{other-2nd-order-formula-hessian}
    \E[(\bz\bz^\top -\bI_n)g(\bz)]=\E[\nabla^2 g(\bz)]
\end{equation}
where $\nabla^2 g$ is the Hessian of $g$. Formulae that relate $\E[\nabla^2 g(\bz)]$
to expectations involving $g(\bz)$ and the density of $\bz$ have been obtained for non-Gaussian
$\bz$, see
\cite[Theorem 1]{janzamin2014score},
\cite{erdogdu2016newton},
\cite[Theorem 2.2]{yang2017learning}
among others. 
In these works, \eqref{other-2nd-order-formula-hessian}
is used to estimate the Hessian in optimization algorithms \cite{erdogdu2016newton},
to estimate the index in some single-index models \cite{yang2017learning}
or to extract discriminative features for tensor-valued data \cite{janzamin2014score}.
If $\bz\sim N({\bf 0},\bI_n)$, the second order iterated Stein formula
from these works reduce to \eqref{other-2nd-order-formula-hessian}.
The fact that iterating Stein's formula introduces 
higher order derivatives should be expected,
because Stein's formula is intrinsically an integration by parts---Stein's formula
is sometimes referred to as \emph{Gaussian integration by parts}, for instance in
\cite[Appendix A.4]{talagrand2010mean} where statistical physics results are rigorously derived
using Stein's formula and the so-called Smart path method as a central tool.
As we will see throughout the paper, the identity for the variance of 
$\bz^\top f(\bz) - \dv f(\bz)$, {given in \eqref{eq1} below,} 
does not involve the second derivatives of $f$ or require their existence,
in striking contrast to the iterated formula \eqref{other-2nd-order-formula-hessian}. 
{Although} both identities can be seen as second order Stein formulae,
{they} are of a different nature {and aim different applications.} 

\subsection{Motivations}
In the Gaussian sequence model where $\by=\bmu+\bep$ is observed
for an unknown mean $\bmu\in\R^n$ and noise $\bep\sim N({\bf 0},\bI_n)$,
Stein's Unbiased Risk Estimate (SURE), proposed in \cite{stein1981estimation},
provides an unbiased estimator for the risk $\|\hbmu-\bmu\|^2$ of any 
{estimator $\hbmu(\by)$ with integrable gradient} as a consequence
of \eqref{stein-first-order-formula}. Namely, the decomposition
$$
\underbrace{\|\by-\hbmu\|^2+2\dv \hbmu -n}_{\SURE} - \|\hbmu-\bmu\|^2
= \bep^\top h(\bep) - \dv h(\bep)
$$
where $h(\bep) = \bep - 2(\hbmu-\bmu)$ shows that $\SURE$ defined above
is an unbiased estimate of $\E[\|\hbmu-\bmu\|^2]$ in virtue of
Stein's formula \eqref{stein-first-order-formula} applied to $h$.
Although the identity $\E[\SURE]=\E[\|\hbmu-\bmu\|^2]$ is well understood 
{and widely used},
little is known about the quality of the approximation $\SURE\approx\|\hbmu-\bmu\|^2$
in probability, or equivalently about the variance and concentration properties
of the random variable $\bep^\top h(\bep)-\dv h(\bep)$ above.
A first motivation of the present paper is to understand the stochastic behavior
of $\bep^\top h(\bep)-\dv h(\bep)$ (beyond unbiasedness) in order to provide
uncertainty quantification for SURE.
If multiple estimators $\{\hbmu^{(1)},\dots,\hbmu^{(m)}\}$ are given for estimation of $\bmu$,
it is natural to use the estimator among those with the {smallest} SURE.
However, statistical guarantees for such SURE-tuned estimate are lacking,
except in very special cases, for instance when $\{\hbmu^{(1)},\dots,\hbmu^{(m)}\}$
are linear functions of $\by$. When the estimators are linear, $\SURE$
is the same as Mallows' $C_p$ \cite{mallows1973some} and the SURE-tuned estimate
enjoys oracle inequalities with remainder term growing with $\log m$,
cf. \cite{bellec2014affine} or \Cref{sec:SURE-oracle-inequality} and the
references therein. We are not aware of oracle inequalities satisfied
by the SURE-tuned estimate for general
non-linear estimators $\{\hbmu^{(1)},\dots,\hbmu^{(m)}\}$; as explained in
\Cref{sec:SURE-oracle-inequality} below, our bounds
on the variance of $\bep^\top h(\bep) - \dv h(\bep)$ are helpful to fill
this gap.

A second motivation of the paper regards the
Lasso $\hbbeta$ in sparse linear regression with Gaussian noise $\bep\sim N({\bf 0},\bI_n)$.
Despite an extensive body of literature on the Lasso in the last two decades,
little is known about the stochastic behavior of the
size of the model selected by Lasso,
for instance no previous results are available on the variance of this integer
valued random variable.
It is well known that the model size of the lasso can be written as $\dv f(\bep)$
for certain function $f:\R^n\to\R^n$,
see \Cref{sec:sparsity} below and the references therein.
This integer valued random variable---the size of the model selected by the Lasso---is
not differentiable with respect to $\bep$ and the Stein formulae
\eqref{stein-first-order-formula}-\eqref{other-2nd-order-formula-hessian}
cannot be iterated further to understand the variance of this discrete random
variable.
Another avenue to bound the variance of a random variable of the form $g(\bz),\bz\sim N({\bf 0},\bI_n)$
is the Gaussian Poincar\'e inequality, which states that $\Var[g(\bz)]\le\E[\|\nabla g(\bz)\|^2]$
for differentiable $g:\R^n\to\R$ \cite[Theorem 3.20]{boucheron2013concentration}.
Again, this is unhelpful if $g(\cdot)$ is discrete-valued and
non-differentiable.

\subsection{Contributions}
The central result of the paper states
that the variance of the random variable $\bz^\top f(\bz) - \dv f(\bz)$
is exactly given by
$$\E\big[(\bz^\top f(\bz) - \dv f(\bz))^2\big] 
= \E\big[ \|f(\bz)\|^2 + \trace\big(\left(\nabla f(\bz)\right)^2\big)\big]
$$
{when both $f:\R^n\to\R^n$ and its derivative are square integrable.} 
Surprisingly, although this is obtained by iterated integration by parts,
the second derivatives do not appear.
As explained in the next section, the second derivatives of $f$ need not exist
for the previous display to hold---in that sense,
the previous display is more broadly applicable than
\eqref{other-2nd-order-formula-hessian}.
The above formula appears especially useful for estimators commonly called for
in high-dimensional statistics such as the Lasso, the Elastic-Net or the
Group-Lasso that are differentiable almost everywhere with respect to the
noise, but not twice-differentiable.
The contributions of the paper are summarized below.

\begin{itemize}
    \item We derive in \Cref{sec:second-order-stein-formula} the above Second Order Stein formula,
        an identity for the variance of the random variable
        $\bz^\top f(\bz)-\dv f(\bz)$.
    \item As a consequence,
        \Cref{sec:variance-divergence} provides an identity and an upper bound
        for {$\Var[\bz^\top f(\bz)-\dv f(\bz) - g(\bz)]$ in general, 
        including} $\Var[\dv f(\bz)]$. 
\end{itemize}
Statistical applications are then provided in \Cref{sec:3}:
\begin{itemize}
    \item
        \Cref{sec:sure-of-sure} leverages the Second Order Stein formula to construct an unbiased risk estimate 
        for Stein's Unbiased Risk Estimate (SURE) in the Gaussian sequence model. 
        We shall call this general method SURE for SURE. 
    \item
        \Cref{sec:confidence-region-SURE} describes confidence regions
        based on SURE and {the Second Order Stein formula.} 
   \item {\Cref{sec:SURE-oracle-inequality}} provides oracle inequalities
        for SURE-tuned estimates. 
    \item
        \Cref{sec:sparsity} provides new bounds on the variance of the size of the model selected by the Lasso
        in sparse linear regression.
    \item \Cref{sec:SOSforLassoEnet} provides SURE for SURE formulas for the Lasso and E-net 
    {as well as} the consistency of SURE for SURE in the Lasso case.
    \item \Cref{sec:debiasing} provides a scheme to de-bias a general class of estimators in linear regression
        where one wishes to estimate a low-dimensional projection of the unknown regression 
        coefficient vector.
    \item \Cref{sec:MonteCarlo} develops a Monte Carlo scheme to approximate 
        the divergence of a general differentiable estimator when the analytic form of the 
        divergence is unavailable. 
\end{itemize}

The above statistical applications are obtained as 
{consequences} 
of the Second Order Stein formula of \Cref{sec:second-order-stein-formula}.
The point of the Second Order Stein formula is that, in many statistical
applications where the random variable $\bz^\top f(\bz) - \dv f(\bz)$ appears,
the standard deviation $\Var[\bz^\top f(\bz) - \dv f(\bz)]^{1/2}$
is of smaller order than either $\dv f(\bz)$ or $\bz^\top f(\bz)$
and the approximation $\bz^\top f(\bz)\approx\dv f(\bz)$
holds not only in expectation as in \eqref{other-2nd-order-formula-hessian},
but also in probability.

\subsection{Notation}
\label{sec:notation}

Throughout the paper, $\|\cdot\|$ is the Euclidean norm 
and $\|\cdot\|_F$ the Frobenius norm, and $f:\R^n\to \R^m$ is L-Lipschitz if 
$\|f(\bu)-f(\bv)\|\le L\|\bu-\bv\|,\ \forall \bu,\bv$. 
For $f:\R^n\to \R^n$ with components $f_i$, denote 
by $\nabla f_i\in\R^n$ the gradient of each $f_i$, and by $\nabla f$
the matrix in $\R^{n\times n}$ with columns $\nabla f_1,...,\nabla f_n$.
The random variable $X$ is said to be in $L_1$ (respectively, $L_2$)
if $\E[|X|] <+\infty$ (resp., $\E[X^2]<+\infty$).
For two symmetric matrices $\bA,\bB$ of the same size, we write $\bA\preceq\bB$
if and only if $\bB-\bA$ is positive semi-definite. 

\section{Mathematical results}
\label{sec:2}

\subsection{A Second Order Stein formula}
\label{sec:second-order-stein-formula}

Stein's formula \eqref{stein-first-order-formula} states that the random variable
\begin{equation}
\bz^\top f(\bz) - \dv f(\bz)
\label{eq:centered-random-variable}
\end{equation}
is mean-zero.
The main result
of the current paper is the following \emph{Second Order Stein} formula,
which provides an identity for the variance of the random variable \eqref{eq:centered-random-variable} 
with $f$ in the Sobolev space $W^{1,2}(\gamma_n)$ with respect to the standard Gaussian measure $\gamma_n$ in $\R^n$. 
Let $C^\infty_0(\R^n)$ be the space of infinitely differentiable 
functions $\R^n\to\R$ with compact support.
The Sobolev space $W^{1,2}(\gamma_n)$ is defined as the completion of
the space $C^\infty_0(\R^n)$ with respect to the norm
\begin{equation}
    \label{norm-sobolev}
    \|g \|_{1,2} = \E[ g(\bz)^2]^{1/2} + 
    \E[\| \nabla g(\bz)\|^2]^{1/2}, \qquad g\in C^\infty_0(\R^n)
\end{equation}
where $\bz\sim N({\bf 0},\bI_n)$.
By \cite[Proposition 1.5.2]{bogachev1998gaussian}, the space
$W^{1,2}(\gamma_n)$ corresponds to locally integrable functions $g:\R^n\to\R$ that are
weakly differentiable and
$\E[g(\bz)^2] + \E[ \|\nabla g(\bz)\|^2]<+\infty$
for $\bz\sim N({\bf 0}, \bI_n)$. 
{This definition of the Sobolev space emphasizes the feasibility of approximating 
the functions in the space by infinitely smooth ones. 
Alternatively, $W^{1,2}(\gamma_n)$ can be defined as the completion of 
the space of Lipschitz functions under the same norm.} 
We refer to Section 1.5 in \cite{bogachev1998gaussian} for a complete description
of the space $W^{1,2}(\gamma_n)$.

\begin{theorem}
    \label{thm:second-order-stein}
    Let $\bz=(z_1,...,z_n)$ be a standard normal $N({\bf 0},\bI_n)$ random vector.
    Let $f_1,...,f_n$ be functions $\R^n\to\R$
    and $f$ be the column vector in $\R^n$ 
    with $i$-th component equal to $f_i$. Assume throughout that each $f_i$ is square integrable, 
    i.e. $\E[f_i(\bz)^2]< \infty$. 

    (i)
    Assume that each $f_i$ is twice continuously differentiable
    and that its first and second order derivatives have sub-exponential growth.
    Then
    \begin{align}
    \E\left[
        (\bz^\top f(\bz) - {\dv }f(\bz))^2
    \right]
    =
    \E \sum_{i=1}^n f_i^2(\bz) 
    + \E\sum_{i=1}^n \sum_{j=1}^n \frac{\partial f_i}{\partial x_j}(\bz) 
    \frac{\partial f_j}{\partial x_i}(\bz).
        \label{eq1}
    \end{align}
    An equivalent version of (\ref{eq1}) 
    for functions of $\by = N(\bmu,\sigma^2\bI)$, 
    in vector notation, is given in (\ref{th-1-sigma}) below.
    
    (ii) If $f:\R^n\to\R^n$ is $L$-Lipschitz for $L<+\infty$, then
    \eqref{eq1} holds.


    (iii) If each component $f_i$ of $f$ 
    belongs to 
    $W^{1,2}(\gamma_n)$, then \eqref{eq1} holds. 
\end{theorem}

We note that while (i) and (ii) may have broader appeal due to their simplicity, 
the assumption in (iii) above is the most general:
If $f$ satisfies either the assumption in (i) or (ii) 
then the components of $f$
are in $W^{1,2}(\gamma_n)$ and the assumption in (iii) holds.
The presentation above {also} highlights the proof strategy:
We first derive \eqref{eq1} under the smoothness assumption in (i)
and (ii)-(iii) then follow by an approximation argument.

Here is a proof of \Cref{thm:second-order-stein}(i). Extensions (ii) and (iii) are proved in \Cref{appendix:proof(ii)-(iii)}.
\begin{proof}[Proof of \Cref{thm:second-order-stein} (i)] 
When the functions $f_i$ and $(\partial/\partial x_j)f_i$ are treated as random variables,
their argument 
    is always $\bz$ through the proof, so we simply write $f_i$ for $f_i(\bz)$
    and similarly for the partial derivatives.
    A sum $\sum_i$ or $\sum_k$ always sums over $\{1,...,n\}$.

    Write the left hand side in \eqref{eq1} as
    \begin{equation}\nonumber
    \E\sum_i\left(z_i f_i -  \frac{\partial f_i}{\partial x_i} \right)
        \left(\sum_j z_j f_j - \sum_l \frac{\partial f_l}{\partial x_l}\right).
    \end{equation}
    By a first application of Stein's formula for each term $z_i$ {below},
    the identity
    \begin{equation}
        \E\left[\Big(z_if_i(\bz) - \frac{\partial f_i}{\partial x_i}(\bz)\Big) g(\bz)\right] 
    = \E\left[f_i(\bz)\frac{\partial g}{\partial x_i}(\bz)\right]
    \label{Stein-formula-applied-componentize-to-g}
    \end{equation}
    holds for any continuously 
    differentiable function $g:\R^n\to\R$ 
        such that the partial derivatives
        $(\partial/\partial x_i)g$ have 
    sub-exponential growth.
    Hence the left hand side in \eqref{eq1} equals
$$
\E \sum_i f_i^2 + \E \sum_i f_i
    \sum_j z_j \frac{\partial f_j}{\partial x_i} - \E \sum_i f_i \sum_l \frac{\partial^2 f_l}{\partial x_i \partial x_l}.
$$
We again apply Stein's formula to each $z_j$ in the second term above to obtain
$$
\E \sum_i f_i^2 
+ \E \sum_i \sum_j \frac{\partial f_i}{\partial x_j} \frac{\partial f_j}{\partial x_i}
+
\E \sum_i f_i \sum_j \frac{\partial^2 f_j}{\partial x_j \partial x_i}
- \E \sum_i f_i \sum_l \frac{\partial^2 f_l}{\partial x_i \partial x_l}.
$$
Since $f$ is twice continuously differentiable, by {the Schwarz theorem} 
on the symmetry of the second derivatives we have
$\sum_j (\pa/\pa x_j)(\pa/\pa x_i) f_j = \sum_\ell (\pa/\pa x_i)(\pa/\pa x_\ell) f_\ell$ 
and the proof of \eqref{eq1} is complete.
\end{proof}

The Second Order Stein formula \eqref{eq1} can then be rewritten as
\bel{SOS}&
\E\Big[(\bz^\top f(\bz) - \dv f(\bz))^2\Big] 
= \E\Big[ \|f(\bz)\|^2 + \trace\big(\left(\nabla f(\bz)\right)^2\big)\Big].
\eel
By the Cauchy-Schwarz inequality,
\begin{align}\label{eq1.5}
    \E\left[
        (\bz^\top f(\bz) - {\dv }f(\bz))^2
    \right]
    &\le
    \E \sum_{i=1}^n f_i^2(\bz)
    +
    \E \sum_{i=1}^n\sum_{j=1}^n \Big(\frac{\partial f_i}{\partial x_j}(\bz)\Big)^2,\\ \nonumber
    &=
    \E\Big[ \|f(\bz)\|^2 + \|\nabla f(\bz)\|_F^2\Big].
\end{align}
If $\nabla f(\bz)$ is almost surely symmetric, then
$\trace((\nabla f(\bz))^2) = \|\nabla f(\bz)\|_F^2$ 
and the above inequality is actually an equality.
However the inequality in (\ref{eq1.5}) is strict otherwise. 

If $\by = \bmu + \bep$ with $\bep\sim N({\bf 0},\sigma^2\bI_n)$ 
and $f:\R^n\to\R^n$
then
\begin{equation}
\label{th-1-sigma} 
    \E[(\bep^\top f(\by) - \sigma^2\dv f(\by))^2] =
    \sigma^2 \E[\|f(\by)\|^2] + \sigma^4\E[\trace\big(\left(\nabla f(\by)\right)^2\big)] 
\end{equation}
is obtained by setting $\bz=\bep/\sigma$ and applying \Cref{thm:second-order-stein} to 
$\tilde f(\bx)=\sigma f(\bmu+\sigma \bx)$, provided that $\tilde f$ 
satisfies the assumption of \Cref{thm:second-order-stein}~(iii).
Similarly if $\by=\bmu+\bep$ with $\bep\sim N({\bf 0},\bSigma)$ then
$${\E}\big[\{\bep^\top f(\by) 
- \trace(\bSigma\,\nabla f(\by))\}^2\big]
 = \E\big[\|\bSigma^{1/2} f(\by)\|^2\big] 
 + \E\trace\big[\{\bSigma\nabla f(\by)\}^2\big] 
$$
follows from \Cref{thm:second-order-stein} with $\bz = \bSigma^{-1/2}\bep$
and $\tilde f(\bx) = \bSigma^{1/2}f(\bmu+\bSigma^{1/2}\bx)$.

\Cref{thm:second-order-stein} is also applicable under the {central limit theorem.}  
Let $\bep_m = m^{-1/2}\sum_{i=1}^m \bx_i$ with iid $\bx_i\in \R^n$, 
$\E[\bx_i]={\bf 0}$ and $\E[\bx_i\bx_i^\top] = \bSigma$. 
If 
$\{\|f(\bep_m)\|^2, \|\nabla f(\bep_m)\|_F^2, (\bep_m^\top f(\bep_m))^2, m\ge 1\}$ is uniformly integrable
and $\nabla f$ is almost everywhere continuous, then 
$\E\big[\bep_m^\top f(\bep_m)\big]
= \E\big[\trace\big(\bSigma\,\nabla f(\bep_m)\big)\big] + o(1)$ and 
\bel{th-1-Sigma}
&& \E\Big[\Big(\bep_m^\top f(\bep_m) 
- \trace\big(\bSigma\,\nabla f(\bep_m)\big)\Big)^2\Big]
\\ \nonumber &=& \E\big[\big\|\bSigma^{1/2} f(\bep_m)\big\|^2\big] 
+ \E\big[\trace\big(\big(\bSigma\nabla f(\bep_m)\big)^2\big)\big] 
+ o(1)
\eel
as $m\to\infty$ for fixed $n$. 
If, for instance, $f:\R^n\to\R^n$ is $L$-Lipschitz with almost everywhere
continuous gradient
and $\E[\|\bx_i\|^{{4}}]\le C$ 
for some constant $C>0$ possibly depending on the
dimension, then
$\{\|\bep_m\|^4,m\ge 1\}$ is {uniformly integrable,}
which implies that
$\{\|f(\bep_m)\|^2, \|\nabla f(\bep_m)\|_F^2, (\bep_m^\top f(\bep_m))^2, m\ge 1\}$
is 
{also} uniformly integrable, so that \eqref{th-1-Sigma} holds
by, e.g.,
\cite[Theorem 1.11.3]{van1996weak}.

    

\subsection{Inner-product structure}
    The two sides of \eqref{SOS} are quadratic in $f$ and
    \eqref{SOS} is endowed with an inner-product struture.
Indeed, if $f,h:\R^n\to\R^n$ satisfy the assumption of \Cref{thm:second-order-stein}(iii),
so do $f+h$ and $f-h$. Appliying \eqref{SOS} to $f+h$ and $f-h$, taking the
difference and dividing by 4 {yield}
\bel{inner-product-structure}
     &&\E\Big[
         \big\{\bz^\top f(\bz) - \dv f(\bz)\big\}
         \big\{\bz^\top h(\bz) - \dv h(\bz)\big\}
     \Big] 
\cr
     &=& 
     \E\Big[ f(\bz)^\top h(\bz) + \trace\big(\nabla f(\bz)\nabla h(\bz)\big)\Big]
\eel
thanks to $\trace\big(\nabla f(\bz)\nabla h(\bz)\big)
=
\trace\big(\nabla h(\bz)\nabla f(\bz)\big)$. 

\subsection{Extensions and upper bounds on variance}
\label{sec:variance-divergence}
The Second Order Stein formula \eqref{eq1}
lets us derive the variance of more general random variables of the form 
\bel{eq:general-random-variable}
\bz^\top f(\bz) - \sigma^2\dv f(\bz) - g(\bz) 
\eel
and a related upper bound for the variance, 
where $\bz\sim N({\bf 0},\sigma^2\bI_n)$ and $f: \R^n\to\R^n$ and $g:\R^n\to\R$ satisfy 
the smoothness and integrability conditions of respective dimensions 
as in \Cref{thm:second-order-stein}. 
The formula for the variance of \eqref{eq:general-random-variable} can be viewed 
as an extension of \eqref{eq1} from $g=0$ to general $g$. 
As a special case, this provides a formula and an upper bound 
for the variance of the divergence with $g(\bx) = \bx^\top f(\bx)$. 

\begin{theorem}
    \label{thm:general-SOS}
    Let $\bz\sim N({\bf 0},\sigma^2\bI_n)$, 
    $f$ be as in \Cref{thm:second-order-stein} and $g\in W^{1,2}(\gamma_n)$. 
    Then the variance of the random variable in \eqref{eq:general-random-variable} satisfies
    \bel{variance-general} 
    && \Var\big(\bz^\top f(\bz) - \sigma^2\dv f(\bz) - g(\bz)\big)
    \cr &=& \E\big[\sigma^2\|f(\bz) - \nabla g(\bz)\|^2+\sigma^4\trace\big(\big(\nabla f(\bz)\big)^2\big)\big] 
    + \tilde V(g)
    \\ &\le& \E\big[\sigma^2\|f(\bz) - \nabla g(\bz)\|^2+\sigma^4\trace\big(\big(\nabla f(\bz)\big)^2\big)\big]. 
    \label{variance-ineq-general}
    \eel
    with $\tilde V(g) =\Var(g(\bz))-\sigma^2\E[\|\nabla g(\bz)\|^2]$.
In particular, with $g(\bx) = \bx^\top f(\bx)$, 
	\bel{prop:variance-divergence}
        \Var[\dv f(\bz)] 
        & = &\E[\trace\left(( \nabla f(\bz))^2\right)] + \Var[ \bz^\top f(\bz)/\sigma^2]
        \cr && - \E[ \|f(\bz)/\sigma\|^2] - 2 \E[ f(\bz)^\top {(\nabla f(\bz))} \bz/\sigma^2]
        \cr &\le& \E[\trace\left(( \nabla f(\bz))^2\right)] + \E[\|{(\nabla f(\bz))} \bz/\sigma\|^2]. 
    \eel
\end{theorem}

\begin{proof} Assume without loss of generality $\sigma=1$. 
{By \eqref{Stein-formula-applied-componentize-to-g},  
\bes
\E\Big[\big(\bz^\top f(\bz) - \dv f(\bz)\big)g(\bz)\Big] = \E\Big[f(\bz)^\top\nabla g(\bz)\Big], 
\ees
so that \eqref{variance-general} follows from \Cref{thm:second-order-stein} and some algebra. 
The} Gaussian Poincar\'e inequality \cite[3.20]{boucheron2013concentration} applied to $g$ reads 
$\tilde V(g)\le 0$ 
which gives \eqref{variance-ineq-general}.
Finally, for $g(\bx) = \bx^\top f(\bx)$, we have $(\partial g/\partial
    x_i)(\bz) = f_i(\bz) + \sum_{j=1}^n z_j (\partial f_j/\partial x_i)(\bz)$
    so that $\nabla g(\bz) = f(\bz) + {(\nabla f(\bz))} \bz$. 
\end{proof}

As we have mentioned earlier, \eqref{variance-general} becomes \eqref{th-1-sigma} when $g=0$. 
The extension in \Cref{thm:general-SOS} 
is particularly useful in our investigation of the variance of SURE and 
other quantities of interest.  
A striking feature of {the upper bound in \eqref{prop:variance-divergence}} 
is that the variance of the random variable
$\dv f(\bz)$, defined using the first order derivatives of $f$,
can be bounded from above using only the first order partial derivatives of $f$.
In particular, the second partial derivatives of $f$ may be arbitrarily large or may not exist.
This feature will be used in the next section to study the variance of the size
of the model selected by the Lasso
in linear regression, which takes the form $\dv f$ for a certain function $f$.

It can be seen from the {definition of $\tilde V(g)$} 
that the inequality \eqref{variance-ineq-general} involves 
a single application of the Gaussian Poincar\'e inequality {to $g(\bz)$.} Thus, it holds with 
equality if and only if $g(\bx)$ is linear.

Finally, we obtain the following by applying \Cref{prop:variance-divergence}
to $g(\bx)=\bx^\top\{f(\bx)-\E[\nabla f(\bz)]\bx\}$: 
\bel{variance-divergence-2}
     &&   \Var[\dv f(\bz)]
   \\&\le&
        \E[\trace\left(( \nabla f(\bz)-\E[\nabla f(\bz)])^2\right)]
        + \E[\|\{\nabla f(\bz)-\E[\nabla f(\bz)] \} \bz/\sigma\|^2].
        \nonumber
\eel
This recovers $\Var[\dv f(\bz)]=0$ when $f(\bz)$ is linear in $\bz$
and may be useful in other cases where $\nabla f(\bz)-\E[\nabla f(\bz)]$ is
small. 

\subsection{Beyond Gaussian distributions}
        Although the next section provides applications
        of \eqref{SOS} {and \eqref{variance-ineq-general}} 
        in the Gaussian case only, we briefly
        mention here extensions to non-Gaussian $\bz$
    when $\bz$ has density $\bx\to\exp(-{\psi}(\bx))$ where ${\psi}:\R^n\to \R$ is
    twice continuously differentiable with Hessian $H(\bx) = \nabla^2 {\psi}(\bx)$.
    Throughout, 
    let $h,f:\R^n\to\R^n$ be smooth vector fields
    with compact support and $g:\R^n\to\R$ be smooth.
    By integration by parts,
    \begin{equation*}
        \int \left\{\nabla {\psi}(\bz)^\top h(\bz) - \dv h(\bz) \right\} g(\bz) e^{-{\psi}(\bz)}\, d\bz
        = \int h(\bz)^\top \nabla g(\bz) e^{-{\psi}(\bz)}\, d\bz.
    \end{equation*} 
    Applying this identity twice, first to $h=f$ and $g(\bx) = \nabla {\psi}(\bx)^\top f(\bx) - \dv f(\bx)$, second to $h(\bx)=\nabla f(\bx)^\top f(\bz)$ and $g=1$ gives
    \bel{SOS-generalization-non-gaussian}
        && 
        \E\big[\{\nabla {\psi}(\bz)^\top f(\bz) - \dv f(\bz) \}^2\big]
          \cr&= &
          \E\big[
              f(\bz)^\top\{H(\bz)\} f(\bz)
            + \trace\{\nabla f(\bz)^2\}
          \big]
    \eel
    which extends \eqref{eq1} to non-Gaussian distributions.
    Similarly to \Cref{prop:variance-divergence}, if ${\psi}$ is additionally strictly
    convex then
    \bel{variance-general-non-gaussian} 
             && \Var\big(\nabla {\psi}(\bz)^\top f(\bz) - \dv f(\bz) - g(\bz)\big)
      \\ \nonumber &=& \E\big[\|H(\bz)^{1/2}f(\bz) - H(\bz)^{-1/2}  \nabla g(\bz)\}\|^2+\trace\big(\big(\nabla f(\bz)\big)^2\big)\big] 
    + \tilde V(g)
    \eel
    where $\tilde V(g) =\Var(g(\bz))-\E[\|H(\bz)^{-1/2}\nabla g(\bz)\|^2]$
    and $\tilde V(g)\le 0$ follows from the Brasecamp-Lieb inequality
    \cite[Theorem 4.9.1]{bakry2013analysis}.
    In the special case where $f=\nabla u$ for some smooth $u:\R^n\to\R$,
    identity \eqref{SOS-generalization-non-gaussian} is related to the Bochner formula
    used in the analysis of diffusion operators \cite[Section 1.16.1]{bakry2013analysis},
    which gives hope to extend \eqref{eq1} and its applications
    to probability measures defined on non-Euclidean manifolds---although this goal
    lies outside of the scope of the present paper. 

\section{Statistical applications}
\label{sec:3}
\subsection{SURE for SURE}
\label{sec:sure-of-sure}
In the Gaussian sequence model, one observes $\by = \bmu + \bep$ 
where the noise $\bep\sim N({\bf 0},\bI_n)$
is standard normal and $\bmu$ is an unknown mean.
Given an estimator $\hbmu(\by)$
where $\hbmu:\R^n\to\R^n$ is some known almost differentiable function 
with $\nabla \hbmu(\by)$ in $L_1$, SURE
provides an unbiased estimate of the mean squared risk $\E \|\hbmu - \bmu\|^2$ given by
\bel{def-SURE}
    \SURE
    =
    \|\hbmu-\by\|^2
    + 2 \dv \hbmu(\by)
    - n.
\eel
The fact that this quantity is an unbiased estimate of $\E \|\hbmu - \bmu\|^2$
is a consequence of the identity
\[
    \|\hbmu - \bmu\|^2 
    =
        \|\hbmu-\by\|^2 + 2 \bep^\top (\hbmu-\bmu) - \|\bep\|^2
\]
with $\E[\|\bep\|^2] = n$
and Stein's formula \eqref{stein-first-order-formula} which asserts that $\E[\bep^\top \hbmu(\by)] = \E[\dv \hbmu(\by)]$
whenever all partial derivatives of $\hbmu$ are in $L_1$.
The random variable $\dv \hbmu(\by)$ can be computed from the observed data since it only involves 
$\by$ as well as the partial derivatives of $\hbmu$.
The quantity 
    $\df = \dv \hbmu(\by)$
is an estimator sometimes 
referred to as the empirical degrees of freedom
of the estimator $\hbmu$. 
In this subsection, we develop Second Order Stein methods to evaluate the accuracy of SURE. 

\subsubsection{SURE for SURE in general}
\label{subsec:SURE4SURE-general}

We define the mean squared risk of the scalar estimator $\SURE$ by
\bel{risk-of-SURE}
    R_\tSURE =
    \E\Big[\Big(\SURE - \|\hbmu - \bmu\|^2\Big)^2\Big].
\eel
This means we treat $\SURE$ as an estimate of the squared prediction error 
$\|\hbmu-\bmu\|^2$ as well. This is reasonable as the actual squared loss $\|\hbmu - \bmu\|^2$ 
is often a more relevant target than its expectation. 
One may also wish to treat $\SURE$ as an estimate of the deterministic 
$\E[\|\hbmu - \bmu\|^2]$ 
and consider the estimation of 
$\Var(\SURE)$.
{Thanks} to \Cref{thm:general-SOS}, we will develop in the next subsection methodologies for 
the consistent estimation of upper bounds for $\Var(\SURE)$ 
and $R_{\tSURE}$ under proper conditions on the gradient of $\hbmu$.

SURE is widely used in practice to estimate 
$\|\hbmu - \bmu\|^2$ or $\E \|\hbmu - \bmu\|^2$
either because it is of interest to estimate the prediction error of $\hbmu$,
or because several estimators of the mean vector $\bmu$
are available and the statistician hopes to
use the $\SURE$ 
{to compare} them on equal footing.
Although $\SURE$ provides an unbiased estimate of 
the loss $\|\hbmu - \bmu\|^2$ and its expectation,
such estimate may end up being not so useful, 
or provide spurious estimates, 
if the quantity \eqref{risk-of-SURE} is too large.
For estimators of interest where $\SURE$ is used in practice, it is important
to understand the risk of $\SURE$ given by \eqref{risk-of-SURE}
in order to provide some uncertainty quantification about the 
{accuracy} 
of $\SURE$.
For instance, one should expect $\SURE$ to be successful if
${R_\tSURE^{1/2}}$ is negligible compared to $\SURE$, i.e.,  ${R_\tSURE^{1/2}} \lll \SURE$.
On the other hand, if ${R_\tSURE^{1/2}} \ggg \SURE$ then we would expect that
estimates from $\SURE$
would be spurious with strictly positive probability and $\SURE$ should not be trusted. 
Under the square integrability condition on the first and second partial derivatives of $f(\by)$, 
\cite{stein1981estimation} proposed an unbiased estimate of the risk (\ref{risk-of-SURE}). 
However, the twice differentiability condition typically fails to hold 
for estimators involving less smooth regularizers such as the Lasso.
\cite{donoho1995adapting} studied 
the performance of SURE optimized separable threshold estimator (SureShrink) and thus 
the accuracy of SURE in this special case. 
\cite{dossal2013degrees} derives an identity for the quantity \eqref{risk-of-SURE}
in the special case of the Lasso. 
In a general study of SURE tuned estimators, 
\cite{tibshirani2018excess} developed a correction for the excess optimism 
with the nominal SURE in such schemes. 
Section 5 in \cite{javanmard2018debiasing}
establishes consistency of SURE for the Lasso 
when the design matrix has iid $N(0,1)$ entries and 
the tuning parameter is large enough.
The Second Order Stein identity 
$$\E\big[\big\{\|\bz\|^2-n + \gamma(\bz)\big\}^2\big] 
=\E\big[ 2p + 2\Delta \gamma(\bz) + \gamma(\bz)^2\big],$$ 
where $\Delta = \sum_{i=1}^n(\pa/\pa x_i)^2$ is the Laplacian, was used in \cite{johnstone1988inadmissibility} to prove the inadmissibility of SURE 
for the estimation of the squared loss of the James-Stein estimator when $n\ge 5$.

The following result, which extends Theorem 3 of \cite{stein1981estimation} to allow application 
to the Lasso and other estimators only one-time differentiable, 
computes the expectation of the quantity \eqref{risk-of-SURE}
as well as an unbiased estimator of it directly through \Cref{thm:second-order-stein}. 

\begin{theorem}
    \label{thm:sure-for-sure}
    Let $\bep\sim N({\bf 0},\bI_n)$ and $\by = \bmu + \bep$.
    Let $\hbmu$ be an estimator of $\bmu$ with $\hbmu:\R^n\to\R^n$
    satisfying the assumptions of \Cref{thm:second-order-stein},
    and define $\SURE$ by \eqref{def-SURE}.
    Then
    \begin{align}\nonumber
      &  \E\left[
            \left(\SURE - \|\hbmu - \bmu\|^2\right)^2
        \right]
    \\  \label{2.4}
    & =
    \E\left[
        4 \|\by - \hbmu(\by)\|^2
        + 4 \trace((\nabla \hbmu(\by) )^2)
        -2n
    \right].
    \end{align}
    Consequently, SURE for SURE
    \bel{SURE-for-SURE}
    {\Rhat_\tSURE} =  4 \|\by-\hbmu(\by)\|^2 + 4 \trace( (\nabla \hbmu(\by))^2 ) - 2n
    \eel
    is an unbiased estimate of the risk of $\SURE$ in \eqref{risk-of-SURE}.
\end{theorem}


\begin{remark} With the notation $f(\by)=\hbmu-\by$ used in \cite{stein1981estimation},
    $\nabla f(\by) = \nabla \hbmu(\by) - \bI_n$ and
SURE for SURE is also given by 
\bel{SFS-alternative}
{\Rhat_\tSURE}  =  2n
        + 4 \|f(\by)\|^2
        + 4 \trace((\nabla f(\by) )^2)
        + 8 \dv f(\by). 
\eel
\end{remark} 

\begin{proof}
    Write $\dv \hbmu$ for $\dv \hbmu(\by)$ and $\hbmu$ for $\hbmu(\by)$.
    By simple algebra
    \bes
        \left(\|\hbmu - \bmu\|^2 - \SURE \right)^2
        &=&
        \big(
            \bep^\top{\{2(\hbmu-\by) + \bep\}}
            +
            (n - 2\dv \hbmu)
        \big)^2
        .
    \ees
    By \eqref{th-1-sigma} applied to
        $\by\to 2\{\hbmu(\by)-\by\} {+} (\by-\bmu)$ 
        we obtain in expectation 
    \bel{SURE-for-SURE-calculation}
    R_{\tSURE}&=&\E[\|2\{\hbmu-\by\}+\bep\|^2+\trace\{(2\nabla \hbmu - \bI_n)^2\}]
      \cr&=&
      \E[4\|\hbmu-\by\|^2+4\trace\{(\nabla \hbmu)^2\}-2n]
      \\&& + \E[4\bep^\top(\hbmu-\by)+4(n-\dv \hbmu)].
      \nonumber
    \eel
    The proof of \eqref{2.4} is complete as the last line
    is 0 in virtue of Stein's formula. 
    Equality \eqref{SFS-alternative} is obtained by observing that
    $f(\by) = \hbmu(\by) - \by$, hence
    \bes
      &  \nabla f(\by) = \nabla \hbmu - \bI_n,
    \quad
    \dv f(\by) = \trace[\nabla \hbmu] - n, \\
&    4 \trace[(\nabla f(\by))^2] + 8 \dv f(\by)
    = 4 \trace[(\nabla \hbmu(\by))^2]
    - 4n
    .
    \ees
\end{proof}

\begin{remark}
    In the Gaussian sequence model where the noise $\bep$ has distribution $N({\bf 0},\sigma^2 \bI_n)$ with $\sigma\ne 1$,
    the estimator $\SURE$ has the form
    \begin{align}\label{SURE-sigma}
    \SURE 
    &= \|\by - \hbmu\|^2 + 2 \sigma^2 \dv \hbmu - \sigma^2 n.
    \end{align}
    Theorem \ref{thm:sure-for-sure} 
     implies that in this setting, 
    SURE for SURE is 
    \bel{SURE4SURE}
    {\Rhat_\tSURE} = 4\sigma^2 \|\by - \hbmu\|^2
        + 4\sigma^4 \trace[ (\nabla \hbmu)^2]
        - 2 \sigma^4n, 
    \eel
    as its expectation is identical to 
    ${R_\tSURE =} \E\big[\big(\SURE - \|\hbmu - \bmu\|^2\big)^2\big]$. 
\end{remark}

\subsubsection{Error bounds and consistency}
\label{subsec:SURE4SURE-Lipschitz}
In the spirit of SURE, SURE for SURE in \eqref{SURE4SURE} 
is fundamentally a point estimator for the risk of SURE as defined in \eqref{risk-of-SURE}. 
It could be the starting point for the construction of an interval estimator 
or simply provide some measure of the performance of SURE when nothing else 
or better is available. 
    Because of its availability in broad settings, \eqref{SURE4SURE} is not expected to always yield 
    sensible and theoretically justifiable interval estimates. 
    Still, we provide here consistent estimation of upper bounds for both the risk 
    \eqref{risk-of-SURE} and $\Var(\SURE)$ under proper conditions on $\hbmu$. 
    Furthermore, when the gradient is a random orthogonal projection as in 
    Lasso, isotonic regression and many other cases, the theorem below proves the 
    consistency of SURE for SURE in \eqref{SURE4SURE}. Define 
    \bel{SURE4SURE-prime} 
    \Rhat'_{\tSURE} = 2\sigma^2\big(\|\by-\hbmu\|^2+\SURE\big),
    \eel
    which satisfies $\Rhat'_{\tSURE} - \Rhat_{\tSURE} 
    = 4\sigma^4\{\dv\hbmu - \trace\big((\nabla\hbmu)^2\big)\}$ compared with SURE for SURE in 
    \eqref{SURE4SURE}. 
In particular,
$\Rhat'_{\tSURE}=\Rhat_{\tSURE}$ if $\nabla
\hbmu$ is a random orthogonal projection. 
    
\begin{theorem}\label{th-consistency}
Let {$\SURE$,} $R_{\tSURE}$, $\Rhat_{\tSURE}$ and $\Rhat'_{\tSURE}$ be as 
in {\eqref{SURE-sigma},} \eqref{risk-of-SURE}, \eqref{SURE4SURE} and \eqref{SURE4SURE-prime}
respectively and assume $\by=\bmu+\bep$ with $\bep\sim N({\bf 0},\sigma^2\bI_n)$. 
\begin{enumerate}
    \item 
If $\E\sum_{i=1}^n|(\nabla\hbmu)_{ii}|<+\infty$, then $\E[\Rhat'_{\tSURE}] \ge \sigma^4 n$.
\item
If the function $\by\to \hbmu(\by)$ is 1-Lipschitz, then 
    \bel{QUARTIC-RISK-SQRT-SURE}
    \E\Big[\Big({(\SURE)_+}^{1/2} - \E[ \|\hbmu - \bmu\|^2]^{1/2}\Big)^4\Big]^{1/4}
    &\le &
    R_{\tSURE}^{1/4} + 3\sigma. 
    \eel
\item
If the function $\by\to \hbmu(\by)$ is 1-Lipschitz and 
$\nabla \hbmu$ is almost everywhere symmetric positive semi-definite, then 
$\Rhat_{\tSURE}\le \Rhat'_{\tSURE}$ and 
\bel{var-SURE}
\quad
\Var(\SURE) = \E\Big[\big(\SURE - \E\big[\|\hbmu-\bmu\|^2\big]\big)^2\Big] \le R_{\tSURE} + \sigma^4 n. 
\eel 
\item
If $\Rhat'_{\tSURE} \ge \Rhat_{\tSURE}$ and 
$\max_{\|\bu\|=1}|\bu^\top(\bI_n-\nabla\hbmu)\bu|\le 1$ 
almost surely, then 
\bel{consistency} 
\E\bigg[\bigg(\frac{\Rhat'_{\tSURE}}{\E[\Rhat'_{\tSURE}]}-1\bigg)^2\bigg] 
\le \frac{16\sigma^4}{\E[\Rhat'_{\tSURE}]}  
\le \frac{16}{n}. 
\eel
\end{enumerate}
\end{theorem}
\begin{remark}
    \label{remark:1-lipschitz-positive-symmetric-penalized-estimators}
    If $\hbmu(\by)=\bX\hbbeta(\by)$ for some fixed $\bX\in\R^{n\times p}$
    and a penalized estimator
    $\hbbeta(\by)=\argmin_{\bb\in\R^p}\{\|\by-\bX\bb\|^2/2 + g(\bb)\}$
    for convex $g:\R^p\to\R$, then $\hbmu$ is 1-Lipschitz
    and $\nabla \hbmu(\by)$ is almost everywhere symmetric
    positive semi-definite. 
    This follows from arguments in
    \cite{bellec2016bounds,bellec2019second_order_poincare} as explained
    at the end of \Cref{sec:appendix-proof-risk-of-sqrt-SURE}.
    Hence for any convex regularized least-squares estimate $\hbbeta(\by)$,
    $\hbmu(\by)=\bX\hbbeta(\by)$ satisfies 
    the conditions in (ii)-(iv) above as well as 
    $\Rhat_{\tSURE}\le\Rhat_{\tSURE}'$ almost surely.
\end{remark}

    If the estimation of the deterministic quantity
    $\E[ \|\hbmu - \bmu\|^2]$ is essential, inequality \eqref{QUARTIC-RISK-SQRT-SURE} 
    asserts that up to an additive absolute constant,
    $R_{\tSURE}$ bounds from above the quartic risk of $\SURE{}^{1/2}$
    when the estimation target is $\E[ \|\hbmu - \bmu\|^2]^{1/2}$.
{If in addition the gradient is almost everywhere symmetric positive semi-definite 
as in \Cref{remark:1-lipschitz-positive-symmetric-penalized-estimators}, 
inequality \eqref{var-SURE} asserts that the risk $\Var(\SURE)$ is bounded from above by 
$R_{\tSURE} + \sigma^4 n\le 2\E[\Rhat'_{\tSURE}]$, with a slight modification $\Rhat'_{\tSURE}$ 
of the SURE for SURE. Moreover, under milder conditions in \Cref{th-consistency} (iv) which 
do not require the symmetry of the random matrix $\nabla\hbmu$, 
$\Rhat'_{\tSURE}$ is a consistent estimator of its expectation. 

When the gradient $\nabla\hbmu$ is a random orthogonal projection, 
$\Rhat'_{\tSURE}$ is identical to $\Rhat_{\tSURE}$, 
so that SURE for SURE is a consistent estimator of its risk $R_{\tSURE}$ 
and $R_{\tSURE}+\sigma^4n \le 2 R_{\tSURE}$ are upper bounds for the risk $\Var(\SURE)$. 
In the lasso case, $\nabla\hbmu$ is an orthogonal projection, cf.
\Cref{sure-lasso} below.

The proof of Theorem \ref{th-consistency}
is given in \Cref{sec:appendix-proof-risk-of-sqrt-SURE}. 
In fact, under the conditions for \eqref{consistency}, we prove the sharper  
\bel{SURE4SURE-var-bd}
\Var\big(\Rhat'_{\tSURE}\big) 
\le 16 \sigma^4\E\big[\Rhat''_{\tSURE}\big] 
\eel
with $\Rhat''_{\tSURE}=(3/4)\Rhat'_{\tSURE} + (1/4)\Rhat_{\tSURE} - \sigma^4\dv\hbmu 
\le \Rhat'_{\tSURE}$. 
This suggests the use of $16\sigma^4\Rhat''_{\tSURE}$ 
or $16\sigma^4\Rhat'_{\tSURE}$ 
as estimated upper bounds for $\Var\big(\Rhat'_{\tSURE}\big)$. 
}

\subsubsection{Difference of two estimators}
\label{subsec:SURE4SURE-diff}

As we have briefly discussed above the statement of Theorem \ref{thm:sure-for-sure}, 
SURE is often used to optimize among different estimators. 
Consider for simplicity the comparison between two estimates 
$\hbmu^{(1)}$ and $\hbmu^{(2)}$ of $\bmu$. In this setting, 
\bel{SFS-diff} 
R_{\tSURE}^{\tdiff} = \E\bigg[\Big(\big\|\hbmu^{(1)}-\bmu\big\|^2 
- \big\|\hbmu^{(2)}-\bmu\big\|^2 - \SURE^{\tdiff}\Big)^2\bigg]
\eel
is the proper risk for SURE, 
{where} 
\bel{SURE-diff}
\SURE^{\tdiff} &=& \SURE^{(1)} - \SURE^{(2)}  
\cr &=& \big\|\hbmu^{(1)}-\by\big\|^2 - \big\|\hbmu^{(2)}-\by\big\|^2 
+ 2\dv\big(\hbmu^{(1)} - \hbmu^{(2)}\big)
\eel
is the difference in SURE between $\hbmu^{(1)}$ and $\hbmu^{(2)}$.  
When the loss $\|\hbmu - \bmu\|^2$ is of smaller order than $n^{1/2}$, 
SURE may produce a spurious estimator due to the estimation of $\|\bep\|^2$ by $n$ 
in (\ref{def-SURE}). However, due to the cancellation of this common chi-square type error, 
the risk of the estimator (\ref{SURE-diff}) could be of smaller order than the risk of SURE 
for both $\hbmu^{(j)}$.
Parallel to Theorem~\ref{thm:sure-for-sure}, the Second Order Stein 
formula leads to the following. 

\begin{theorem}
    \label{thm:sure-for-sure-diff}
    Let $\bep\sim N({\bf 0},\bI_n)$, $\by = \bmu + \bep$, and 
    {$\hbmu^{(1)}$ and $\hbmu^{(2)}$} be estimates of $\bmu$ based on $\by$. 
    Let $\SURE^{\tdiff}$ 
    and $R_{\tSURE}^{\tdiff}$ be as in {(\ref{SURE-diff}) and (\ref{SFS-diff}) respectively} 
    and $f(\by) = \hbmu^{(1)} - \hbmu^{(2)}$.  
    Suppose $f:\R^n\to\R^n$ satisfies the assumptions of \Cref{thm:second-order-stein}. 
    Then, 
    \bel{th-2-2-1}
    R_{\tSURE}^{\tdiff} 
    = \E\Big[ 4 \|f(\by)\|^2 + 4 \trace\Big( \big(\nabla f(\by)\big)^2\Big) \Big]. 
    \eel
    Consequently, SURE for SURE, given by 
    \bel{th-2-2-2}
    	\Rhat_{\tSURE}^{\tdiff} 
	=  4 \|f(\by)\|^2 + 4 \trace\Big( \big(\nabla f(\by)\big)^2\Big), 
    \eel
    is an unbiased estimate of the risk of $\ \SURE^{\tdiff}$ in \eqref{SFS-diff}.
\end{theorem}

\begin{proof} By algebra, 
\bes
\big\|\hbmu^{(1)}-\bmu\big\|^2 - \big\|\hbmu^{(2)}-\bmu\big\|^2 - \SURE^{\tdiff} 
= 2 \bep^{\top}f(\by) - 2\dv f(\by). 
\ees
The conclusion follows directly from \Cref{thm:second-order-stein}. 
\end{proof}

\subsection{Confidence region based on $\SURE$}
\label{sec:confidence-region-SURE}
{While SURE for SURE 
provides an unbiased point estimator for the (mean) squared
difference between $\SURE$ and the squared loss $\|\hbmu - \bmu\|^2$, we may also use 
the Second Order Stein formula to derive interval estimates for 
$\|\hbmu - \bmu\|^2$ {and $\E\big[\|\hbmu - \bmu\|^2\big]$} based on $\SURE$. 
As we are not compelled to directly use the ${\Rhat_\tSURE}$ in \eqref{SURE4SURE} 
to construct such interval estimates, 
we present the following simpler approach. 

\begin{theorem}
    \label{thm:sure-for-sure-tail}
    Let $\by$, $\bmu$, $\hbmu = \hbmu(\by)$ and $\SURE$ be as in 
    \eqref{SURE-sigma}. Then, 
    \bel{th-2-2-1a}
    && \E\Big[\big(\SURE - \|\bmu - \hbmu\|^2 - \|\bep\|^2 + \sigma^2 n\big)^2\Big]
    \\ \nonumber &=& 4\sigma^2 \E\big[ \|\hbmu - \bmu\|^2\big] 
    + 4\sigma^4\E\big[\trace((\nabla\hbmu)^2)]. 
    \eel
    If the right-hand side of \eqref{th-2-2-1a} is bounded by $\sigma^4v_0^22 n\eps_n$ 
    with a constant $v_0$, then
    \bel{th-2-2-1b}&
    \P\Big\{\big|\SURE - \|\bmu - \hbmu\|^2\big| \le \sigma^2 (v_{\alpha} + v_0)\sqrt{2n}\Big\} 
    \ge 1 - \alpha - \eps_n 
    \eel
    for all $\alpha\in (0,1)$, where $v_{\alpha}$ is defined by 
    $\P\{ (2n)^{-1/2}|\chi^2_n - n| > v_\alpha\} = \alpha$, and 
    \bel{th-2-2-1c}&
    \P\Big\{\|\bmu - \hbmu\|^2 \le \SURE + \sigma^2 (v_{-,\alpha} + v_0)\sqrt{2n}\Big\} 
    \ge 1 - \alpha - \eps_n, 
    \eel
    where $v_{-,\alpha}$ is defined by 
    $\P\{ (2n)^{-1/2}(n-\chi^2_n) > v_{-,\alpha}\} = \alpha$. 
    \end{theorem} 
While the left-hand side of \eqref{th-2-2-1a} is quartic in $\|\hbmu-\bmu\|$, 
the right-hand side is quadratic. Thus, $\SURE$ provides an accurate estimate of 
$\|\bmu - \hbmu\|^2$ when the squared error is of greater order than 
$|\|\bep\|^2 - \sigma^2 n| \approx \sigma^2(2n)^{1/2}|N(0,1)|$, provided that 
the second term on the right-hand side of \eqref{th-2-2-1a} is of no greater 
order than $\max\big\{\sigma^4n,\sigma^2 \E\big[ \|\hbmu - \bmu\|^2\big]\big\}$.   
Specifically, in such scenarios, \eqref{th-2-2-1b} implies that $\SURE$ is within 
a small fraction of $\|\bmu - \hbmu\|^2$ when $\sqrt{\|\bmu - \hbmu\|^2/n}$ is of 
greater order than $\sigma n^{-1/4}$, and \eqref{th-2-2-1b} and \eqref{th-2-2-1c} 
provide confidence regions for the entire vector $\bmu$. 
As $\sigma n^{-1/4}$ is known to be a lower bound for the error in the 
estimation of the average loss in the estimation of $\bmu$   
\cite{li1989honest,nickl2013confidence}, 
\eqref{th-2-2-1b} implies the rate optimality of the upper
bound \eqref{th-2-2-1c} for the squared estimation error $\|\bmu - \hbmu\|^2$ and thus the rate
optimality of the resulting confidence region for $\bmu$.

\begin{corollary}
    \label{corollary:exact-quantile-confidence-region}
    If for some sequence $(\gamma_n)$ with $\gamma_n\to 0$
    \begin{equation}
        \label{eq:assum-gamma_n-to-0}
        \E[4\|\hbmu-\bmu\|^2/(n\sigma^2) + 4\trace\{(\nabla \hbmu)^2\})/n]
        \le \gamma_n,
    \end{equation}
    then for all fixed $\alpha\in(0,1)$ independent of $n$, 
\bes
\P\Big\{\big|\SURE - \|\bmu - \hbmu\|^2\big| \le \sigma^2 (v_{\alpha} )\sqrt{2n}\Big\} 
 &\to& (1 - \alpha),
\\
\P\Big\{\|\bmu - \hbmu\|^2 \le \SURE + \sigma^2 (v_{-,\alpha})\sqrt{2n}\Big\} 
&\to & (1 - \alpha)
\ees
where $v_\alpha,v_{-,\alpha}$ are defined in \Cref{thm:sure-for-sure-tail}.
\end{corollary}

{Similar to \Cref{thm:sure-for-sure-tail}, 
\Cref{corollary:exact-quantile-confidence-region} provides} 
    confidence regions for the entire vector $\bmu$ with exact
asymptotic quantiles.
Under the condition \eqref{eq:assum-gamma_n-to-0},
$\SURE$ incurs an error characterized by the quantiles of the
random variable $\|\bep\|^2-\sigma^2n$.

We will verify in 
Theorems~\ref{thm:variance-size-model-lasso} and \ref{thm:upperbound-expected-sparsity-RE} 
that the condition \eqref{eq:assum-gamma_n-to-0}
holds for the Lasso under commonly imposed regularity conditions in sparse
regression theory. 

Alternatively to \Cref{corollary:exact-quantile-confidence-region},
the following result replaces the condition on $\gamma_n$
in \eqref{eq:assum-gamma_n-to-0} by a data-driven surrogate $\hat\gamma_n$ 
{and provides non-asymptotic interval estimates 
for both $\E[\|\hbmu-\bmu\|^2]$ and $\|\hbmu-\bmu\|^2$. 

\begin{theorem}
    \label{thm:data-driven-gamma_n}
    Suppose $\hbmu$ is 1-Lipschitz and $\nabla \hbmu(\by)$ is almost surely symmetric
    positive semi-definite. Let $\SURE$ be as in \eqref{SURE-sigma}. Then, 
    \bel{th6-0}
     \quad && \Var\Big(\SURE-\E\big[\|\bmu - \hbmu\|^2\big] + (n\sigma^2-\|\bep\|^2)\Big)
     \le 4 \E\big[\sigma^2\SURE +\sigma^4 \df\big], 
    \eel
    where $\df = \dv\hbmu$. 
    Moreover, there exist non-negative random variables $X_n$ and $Y_n$ with 
    $\E[X_n^2]=\E[Y_n^2]=1$
    such that
    \bel{th6-1}
     && \big|\SURE-\E\big[\|\bmu - \hbmu\|^2\big] + (n\sigma^2-\|\bep\|^2)\big|\big/\big(\sigma^2\sqrt n\big)
     \cr &\le& X_n\left[\hat\gamma_n^{1/2} + 4Y_n/n^{1/2} + \sqrt{4Y_n}(6/n)^{1/4}\right],
    \eel
    where $\hat\gamma_n = 4\{\SURE/(n\sigma^2)+\trace[\nabla \hbmu]/n\}_+$. 
    Consequently for positive $\alpha,\beta_1$ and $\beta_2$ with 
    $\alpha + \beta_1+\beta_2 = \delta <1$ 
    and $\kappa_{n,\beta_2}=2(6/(\beta_2n))^{1/4}+4/(\beta_2n)^{1/2}$, 
    \bel{th6-2}
    && \P\bigg\{
    \frac{|\SURE-\E[\|\bmu - \hbmu\|^2]|}{\sigma^2 \sqrt{2n}}
    \le v_{\alpha} + \frac{\hat\gamma_n^{1/2}+\kappa_{n,\beta_2}}{(2\beta_1)^{1/2}}\bigg\}
    \ge 1 - \delta.
    \eel
    Moreover, \eqref{th6-1} and \eqref{th6-2} still hold 
    with $\E\big[\|\bmu - \hbmu\|^2\big]$ replaced by 
    $\|\bmu - \hbmu\|^2$.
\end{theorem}

{As discussed in \Cref{remark:1-lipschitz-positive-symmetric-penalized-estimators},}
the conclusions of \Cref{thm:data-driven-gamma_n}} 
are applicable to all convex penalized estimators
in linear regression.
Note that $\hat\gamma_n$ is a biased estimate of the left hand side
of \eqref{eq:assum-gamma_n-to-0} {in general} due to
$\trace[\{\nabla \hbmu\}^2]\le \trace \nabla \hbmu$.
The estimate $\hat\gamma_n$ involves $\df = \trace[\nabla \hbmu]$ instead
of $\trace[\{\nabla \hbmu\}^2]$ because controlling the variance
of $\df$ easily follows from \Cref{sec:variance-divergence},
while we are not aware of available tools to bound the variance of
$\trace[\{\nabla \hbmu\}^2]$
except in specific cases where $\nabla \hbmu$ is a projection. The proofs of 
\Cref{thm:sure-for-sure-tail}, \Cref{corollary:exact-quantile-confidence-region} 
and \Cref{thm:data-driven-gamma_n} are given in \Cref{proof:confidence-regions-SURE}.

\subsection{Oracle inequalities for SURE-tuned estimates}
\label{sec:SURE-oracle-inequality}

Beyond pairwise comparisons, the following result
provides guarantees on the SURE-tuned estimate $\tbmu$,
which is obtained by selecting the estimator
among
$\{\hbmu^{(1)},...,\hbmu^{(m)}\}$ with {the} smallest $\SURE$, i.e., 
\bel{SURE-tuned}
\tbmu = \hbmu^{(\hat k)} 
\ \hbox{ with }\ \hat k = \argmin_{j\in [m]}\,\SURE{}^{(j)}, 
\eel
where
$\SURE{}^{(j)}=\|\hbmu^{(j)} - \by\|^2+2\sigma^2\trace \nabla \hbmu^{(j)} -n\sigma^2$
{for} $\bep\sim N({\bf 0}, {\sigma^2\bI_n})$.

\begin{theorem}
    \label{THM:ORACLE-INEQ}
    Consider the sequence model $\by = \bmu + \bep$
    with $\bep\sim N({\bf 0},{\sigma^2} \bI_n)$.
    Let $\hbmu^{(1)}(\by),\ldots,\hbmu^{(m)}(\by)$ be all $L$-Lipschitz 
    functions of $\by$, $\tbmu$ the SURE tuned estimator in \eqref{SURE-tuned}, 
    $j_0= \argmin_{j=1,...,m} \E \|\hbmu^{(j)} - \bmu\|$, 
    $s^*=\max_{k\in[m]}\E\big[\trace\big((\nabla\hbmu^{(k)} - \nabla\hbmu^{(j_0)})^2\big)\big]$. Then
    
    (i) For any $\alpha\in(0,1)$, with probability at least $1-\alpha$,
    \begin{equation}
        \label{eq:iracle-ineq-SURE-claim-i}
      {\|\tbmu - \bmu\|} - {\|\hbmu^{(j_0)}-\bmu\| }
    \le
    \sigma\max\big\{(8s^*m/\alpha)^{1/4}, (8m(\sqrt 2 L +1)/\alpha)^{1/2} \big\}
    \end{equation}

    (ii) For any $\delta\in (0,1)$, with probability at least $1-\delta$,
    \begin{equation}
        \label{th-37-subgaussian-bound}
        \textstyle
        {\|\hbmu^{(j_0)}- \bmu\|}
                - \min_{j\in[m]}{\|\hbmu^{(j)}-\bmu\| }
        \le
         2L\sigma\sqrt{2\log(m/\delta)}
    \end{equation}
    so that the sum of \eqref{eq:iracle-ineq-SURE-claim-i} and
    \eqref{th-37-subgaussian-bound} provide high probability
    bounds
    on ${\|\tbmu - \bmu\|} - \min_{j\in[m]}{\|\hbmu^{(j)}-\bmu\| }$.

    (iii)
    For some absolute constant $C>0$, in expectation
    \begin{equation}
        \label{eq:oracl-ineq-sure-claim-expectation}
        \E\big[\big({\|\tbmu - \bmu\|}
            - \min_{j\in[m]}{\|\hbmu^{(j)}-\bmu\| }
            \big)^2\big]^{1/2}
        \le
        C\sigma
        \big[
        (s^*m)^{1/4}
        +
        (1 + L)
        m^{1/2}
    \big].
    \end{equation}
    
    (iv) If $\max_{j\in[m]}\E[\|\hbmu^{(j)}-\bmu\|^2]/n \le L\sigma^2$
    then the squared risk enjoys
\begin{equation}
    \label{eq:oracle-ineq-SURE-squared-risk}
    \E[\|\tbmu-\bmu\|^2] - 
    \min_{j\in[m]}
    \E[\|\hbmu^{(j)}-\bmu\|^2]
    \le L \sigma^2 (32 n m)^{1/2}.
    \end{equation}
\end{theorem}

The proof is given in \Cref{sec:proof:oracle-ineq-SURE}.
The assumption that the estimators $\hbmu^{(j)}$ are $L$-Lipschitz functions
of $\by$ is mild, cf. \Cref{remark:1-lipschitz-positive-symmetric-penalized-estimators}. 
Under this assumption, $s^*\le4L^2n$ and
{\eqref{eq:oracl-ineq-sure-claim-expectation}
implies} 
\begin{equation}
    \label{eq:oracle-ineq-SURE-m-n-1/4}
    \frac{{\|\tbmu - \bmu\|}}{n^{1/2}}
            - \min_{j\in [m]}
            \frac{{\|\hbmu^{(j)}-\bmu\| }}{n^{1/2}}
    \le
    C
    \sigma
    (1+L)
    \max\Big\{\Big(\frac{m}{\alpha n}\Big)^{1/4}, \Big(\frac{m}{n\alpha}\Big)^{1/2} \Big\}
\end{equation}
with probability $1-\alpha$
for some absolute $C>0$, where we used the $n^{-1/2}$ scaling
to feature the normalized prediction risk $\|\tbmu-\bmu\|^2/n$.
\Cref{THM:ORACLE-INEQ} (i) can also be understood in terms of sample size
requirement: If $\epsilon>0$ is a fixed precision target and $\alpha\in(0,1)$ then
$n \gtrsim \epsilon^{-2} \max\{m/\alpha, (s^*m/\alpha)^{1/2}\}$ 
samples are sufficient to ensure
$$
\textstyle
\P\big\{
n^{-1/2}\|\tbmu - \bmu\|
-
\min_{j\in[m]}
n^{-1/2} \|\hbmu^{(j)} - \bmu\|
\le \sigma \eps\big\} \ge 1-\alpha
.$$
We are not aware of a previous result of this form
that applies with the above level of generality,
i.e., with no restriction on the nature of the estimators
$\{\hbmu{}^{(1)},...,\hbmu{}^{(m)}\}$ beyond the Lipschitz requirement.
As shown in the next proposition, the dependence in $s^*$
in the term $(s^*m/\alpha)^{1/4}$ of
\eqref{eq:iracle-ineq-SURE-claim-i}
and the dependence in $n$ in the term $(m/(n\alpha))^{1/4}$ 
of \eqref{eq:oracle-ineq-SURE-m-n-1/4}
are unimprovable without additional assumptions. 
\begin{proposition}
    \label{prop:s-star-cannot-removed-for-sure-TUNED}
    There exist absolute constants $C_1,C_2,C_3>0$ such that
    for any $n\ge C_1$, there exist $\bmu\in\R^n$ and two estimators $\hbmu^{(1)}(\by),\hbmu^{(2)}(\by)$ that are
    1-Lipschitz functions of $\by$ such that $s^*= n$ and $\tbmu$ in \eqref{SURE-tuned}
    satisfies 
    $$\P\big\{ {\|\tbmu - \bmu\|} - \min_{j=1,2}{\|\hbmu^{(j)}-\bmu\| }
    \ge
    C_2 \sigma n^{1/4} \big\}\ge C_3.$$
\end{proposition}

Oracle inequalities stronger than \eqref{eq:oracle-ineq-SURE-m-n-1/4}
are available if the estimators $\{\hbmu^{(j)}, j=1,...,m\}$
are affine in $\by$, i.e., of the form $\hbmu^{(j)}=\bA_j\by + \bb_j$
for deterministic $\bA_j\in\R^{n\times n}$ and $\bb_j\in\R^n$.
In this case,
$\SURE{}^{(j)}=\|\hbmu^{(j)} - \by\|^2+2\sigma^2\trace \bA_j -n\sigma^2$
reduces to Mallows $C_p$ \cite{mallows1973some} and if $\tbmu$ is the estimate among 
$\{\hbmu^{(j)}, j\in[m]\}$ 
with the smallest $\SURE$, then
\begin{equation}
    \label{eq:oracle-inequality-C_p-linear-estimators}
\E[\|\tbmu - \bmu\|
-
\min_{j\in[m]}
\|\hbmu^{(j)} - \bmu\|
]
\le C (1+L)\sigma(\log m )^{1/2} 
\end{equation}
for some absolute constant $C>0$, provided that the operator norm of $\bA_j$ is
less than $L$ for all $j\in[m]$, cf.  \cite[Proposition 3.1]{bellec2014affine} or
\cite{arlot2009data,tibshirani2018excess} for related results.
If $\bb_j=0$ for all $j$ and the matrices $\bA_j$ are symmetric and totally
ordered in the sense of positive semi-definite matrices, the right hand side in
\eqref{eq:oracle-inequality-C_p-linear-estimators} can even be reduced to
$O(\sigma)$ \cite{kneip1994ordered,bellec2019cost}.
Equation \eqref{eq:oracle-inequality-C_p-linear-estimators} provides an oracle
inequality with respect to the risk $\|\hbmu-\bmu\|$; oracle
inequalities of form \eqref{eq:oracle-ineq-SURE-squared-risk}
with respect to the squared risk are studied in
\cite{leung2006information,dalalyan2012sharp} using exponential weights
procedures, and in \cite{dai2012deviation,dai2014aggregation} \cite[Theorem
2.1]{bellec2014affine} using a convex relaxation of $\SURE$ named
$Q$-aggregation and introduced in \cite{rigollet2012kullback}.
For linear estimators, these works exhibit an error term involving $\log m$
thanks to the availability of the Hanson-Wright inequality, which provides
exponential concentration bounds for random variables of the form
$\bep^\top(\bA_j-\bA_k)\bep$.

As the optimal remainder term for an oracle inequality of the form
\eqref{eq:oracle-ineq-SURE-squared-risk} satisfied by any estimator
of the form $\hbmu^{(\hat j)},\hat j\in[m]$
is of order $\sigma^2(n \log m )^{1/2}$
for deterministic vectors $\hbmu^{(j)}=\bb_j$
\cite[Theorem 2.1]{rigollet2012sparse}, {the} dependence
in $n$ of \eqref{eq:oracle-ineq-SURE-squared-risk} is optimal 
when $m$ is smaller than an absolute constant.

Compared to these existing results,
the novelty of \Cref{THM:ORACLE-INEQ}
lies in its scope: \Cref{THM:ORACLE-INEQ} and \eqref{eq:oracle-ineq-SURE-m-n-1/4}
are applicable
to to any collection of $L$-Lipschitz non-linear estimators $\{\hbmu^{(j)},j\in[m]\}$.
The $L$-Lipschitz assumption is mild---cf. \Cref{remark:1-lipschitz-positive-symmetric-penalized-estimators}---and
\Cref{THM:ORACLE-INEQ} is thus
applicable far beyond the case of linear
estimators studied in the
aforementioned literature.

A drawback of \Cref{THM:ORACLE-INEQ} compared to \eqref{eq:oracle-inequality-C_p-linear-estimators} is the sub-optimality of the dependence in $m$. The difficulty to obtain a rate logarithmic in $m$ is due to the unavailability of exponential concentration equalities for random variables of the form
$\bep^\top(\hbmu^{(j)}-\hbmu^{(k)})$ for non-linear $\hbmu^{(j)},\hbmu^{(k)}$.

Finally, we note that an oracle inequality for the squared risk
can be obtained using the Q-aggregation procedure
from \cite{rigollet2012kullback,dai2012deviation,dai2014aggregation}
and its analysis in \cite{bellec2014affine}. Indeed,
let $\tbmu_Q = \sum_{j=1}^m \hat\theta_j \hbmu^{(j)}$ where
$$
\hbtheta=\argmin_{\btheta\ge{\bf 0}, \,{\bf 1}^\top\btheta=1}
{\bigg[} 
\|\hbmu_{\btheta} - \by\|^2+2\sigma^2
\sum_{j=1}^m \theta_j \dv \hbmu^{(j)}
+\frac 1 2
\sum_{j=1}^m\theta_j\|\hbmu^{(j)}-\hbmu_{\btheta}\|^2{\bigg]}
$$
and $\hbmu_{\btheta}=\sum_{j=1}^m\theta_j\hbmu^{(j)}$
for every $\btheta$ in the simplex $\{\btheta\in\R^m:
\btheta\ge{\bf 0}, \,{\bf 1}^\top\btheta=1 \}$
in $\R^m$. Then $\tbmu_Q$ satisfies
under the assumptions of \Cref{THM:ORACLE-INEQ}
\begin{equation}
    \label{eq:Q-aggregation-oracle-inequality}
    \E[\|\tbmu_Q - \bmu\|^2] - \min_{k\in[m]}\E[\|\hbmu^{(k)}-\bmu\|^2]
\le 
2\sigma^2\Big(
    \sqrt{s^*m} + (1+L)^2m
\Big)
+\sigma^2L^2.
\end{equation}
We refer to \cite{dai2012deviation,dai2014aggregation,bellec2014affine} for
details on the construction of $\tbmu_Q$.
The proof of \eqref{eq:Q-aggregation-oracle-inequality} is given
at the end of \Cref{sec:proof:oracle-ineq-SURE}.

\subsection{The variance of the model size of the Lasso}
\label{sec:sparsity}
Consider a linear regression model
\begin{equation}
    \by = \bX \bbeta + \bep,
\label{LM-lasso}
\end{equation}
where $\bbeta$ is the true coefficient vector,
$\bep\sim N({\bf 0},\sigma^2\bI_n)$ is the noise
and $\bX$ is a deterministic design matrix.
Consider the Lasso which solves the optimization problem
\begin{equation}
    {\lasso} = \argmax_{\bb\in\R^p}
    \big\{\|\bX\bb - \by\|^2/(2n) + \lambda \|\bb\|_1 \big\}.
    \label{lasso}
\end{equation}
Let $\Shat=\{j\in[p]: ({\lasso})_j \ne 0\}$ be the support of the Lasso.
We are interested in the size of $\Shat$ denoted by $|\Shat|$.
Even though the Lasso and sparse linear regression
have been studied extensively 
in the last two decades, little is known about the stochastic
behavior of the discrete random variable $|\Shat|$.
Under the sparse Riesz or similar conditions, $|\Shat| \lesssim \|\bbeta\|_0$ with high probability 
\cite{ZhangH08, zhang10-mc+, zhang2012general} but such results only
imply a bound of the form $\Var[|\Shat|] \lesssim \|\bbeta\|_0^2$ on the
variance; we will see below that the variance of $|\Shat|$ is typically much smaller. 
There are trivial situations where the behavior of $|\Shat|$
is well understood: if $\lambda$ is very large for instance, then $|\Shat|=0$
with high probability.
Or, under strong conditions on $\bX$ and $\bbeta$ that grants support recovery
(cf. for instance, the conditions given in 
\cite{meinshausen2006high, zhao2006model, tropp2006just, Wainwright09}),
$\Shat = \supp(\bbeta)$ holds with probability at least $1-1/p^2$ and in this case 
$\Var[|\Shat|]\le\E[(|\Shat|-s_0)^2] \le 1$. 

Outside of these situations, studying $|\Shat|$ appears delicate;
for instance, our previous attempts at studying the variance of $|\Shat|$
went as follows. Let $(\be_1,...,\be_p)$ be the canonical basis in $\R^p$
and let $\bx_j=\bX\be_j$ for all $j=1,...,p$.
The KKT conditions of the Lasso are given by
$$\bx_j^\top(\by - \bX\lasso)/(n\lambda)
\begin{cases}
    = \sgn(({\lasso})_j) & \text{ if } ({\lasso})_j\ne 0,\\
    \in [-1,1] & \text{ if } ({\lasso})_j=0. 
\end{cases}
$$
At a given point $\by$, to understand the stability of $\Shat$,
a natural avenue is to identity how close
the quantities $\bx_j^\top(\by - \bX{\lasso})/(n\lambda)$
are from $\pm1$ for the indices $j\notin \Shat$.
If many indices $j\notin\Shat$ are such that
$\bx_j^\top(\by - \bX{\lasso})/(n\lambda)$ is extremely close to $\pm 1$,
then a tiny variation in $\by$ may push some of the quantities $\bx_j^\top(\by - \bX{\lasso})/(n\lambda)$ towards $\pm 1$
resulting in many new variables entering the support for this tiny variation in $\by$.
The current model size $|\Shat|$ is non-informative about how many indices $j\notin\Shat$ are such that
$\bx_j^\top(\by - \bX{\lasso})/(n\lambda)$ is extremely close to $\pm 1$
and the random variable $|\Shat|$ appears prone to instability.

With the Second Order Stein formula \eqref{eq1} and the tools developed in 
the previous section,
the variance of $|\Shat|$ can be bounded as follows.
First, we need to describe a condition on the deterministic matrix $\bX$
which ensures that the KKT conditions of the Lasso hold strictly with probability 1.
We say that the KKT conditions hold strictly if
\begin{equation}
    \forall j\notin\Shat, \qquad
    -1
< \frac{1}{\lambda n} \bx_j^\top (\by - \bX{\lasso}) < 1.
\label{KKT-strict}
\end{equation}

\begin{assumption}
    \label{assum:bX}
    For all $\delta_1,...,\delta_p\in\{-1,1\}$ and $1\le j_0<j_1<\cdots<j_n\le p$, 
    $$
    \rank\begin{pmatrix} \bx_{j_0} & \bx_{j_1} & \cdots & \bx_{j_n} \cr 
    \delta_{j_0} & \delta_{j_1} & \cdots & \delta_{j_n} \end{pmatrix}_{(n+1)\times(n+1)} = n+1. 
    $$
\end{assumption}


\begin{proposition}\label{prop-4-1}
    If $\bX$ satisfies the above assumption 
    then the set $B=\{j\in[p]: |\bx_j^\top(\by - \bX{\lasso})| = \lambda n\}$
    is such that $\bX_B$ has rank $|B|$
    and the solution ${\lasso}$ to the optimization problem \eqref{lasso} is unique.
    Furthermore, if $\P\big[\bv^\top \bep = c\big]=0$ for all vectors $\bv\neq{\bf 0}$ and real $c$,
    then the KKT conditions of the Lasso $\lasso$ 
    hold strictly 
    with probability 1,
    i.e., \eqref{KKT-strict} holds with probability 1.
\end{proposition}


Expositions of the results in the first part of the above proposition 
exist in the literature, see for instance \cite[Section 3]{zhang10-mc+} or 
\cite{tibshirani2012,tibshirani2013lasso}. 
Compared with previous versions of the condition on the design, 
\Cref{assum:bX}, which clearly holds with probability 1 when 
$\bX$ is the realization of a continuous distribution over $\R^{n\times p}$,   
gives a natural interpretation in terms of the rank of specific matrices.
The fact that the KKT conditions of the Lasso hold strictly with probability one
is known although it is difficult to pinpoint an existing result 
{in the form of \Cref{prop-4-1}} 
in the literature.
We provide a short proof in \Cref{sec:proof-prop-4-1} for completeness. 

Next define the function $f:\R^n\to\R^n$ by
\begin{equation}
f:\bep \to \bX ({\lasso} - \bbeta).
\label{def-f}
\end{equation}
Then the function $f$ is 1-Lipschitz
and this property holds true for all convex penalized Least-Squares
estimators \cite[Proposition 3]{bellec2016bounds}.
Consequently, almost everywhere the partial derivatives of $f$ exist
and $\nabla f(\bep)$ has operator norm at most one.
It is enough to compute the gradient of $f$ Lebesgue almost everywhere
and by the above {proposition,} the KKT conditions holds strictly
for almost every point $\bep_0\in\R^n$.

If the KKT conditions of the Lasso hold
strictly for $\bep_0$, then by Lipschitz continuity of $\bep\to\bX{\lasso}$
the KKT conditions also hold strictly in small enough nontrivial neighbourhood of $\bep_0$.
In this small neighborhood, the sign and support of ${\lasso}$ are unchanged
and we have for $\|\bh\|$ small enough
$$
\bX {\lasso}(\bep_0 + \bh) = 
\bX_{\Shat} (\bX_{\Shat}^\top \bX_{\Shat})^{-1}(
    \bX_{\Shat}^\top ( \bep_0 +{\bh} + \bX\bbeta) - \lambda n {\rm{sign}}({\lasso}({\bep_0}))
)
$$
where $\Shat$ denotes the locally constant support equal to the support of 
${\lasso}(\bep_0)$. 
In this neighbourhood 
the map $\bh\to \bX{\lasso}(\bep_0 + \bh)$ as well as the map $\bh\to\bX({\lasso}(\bep_0 + \bh) - \bbeta)$
are locally affine with linear part equal to the orthogonal projection
\bel{diff-Lasso}
\bP_{\Shat} = \bX_{\Shat} (\bX_{\Shat}^\top \bX_{\Shat})^{-1}\bX_{\Shat}^\top.
\eel
We conclude this calculation with the following lemma. 

\begin{proposition}\label{prop-diff-lasso} 
Let ${\lasso}$ be the Lasso estimator (\ref{lasso}) with data $(\bX,\by)$ satisfying 
$\by = \bX\bbeta+\bep$. Define $f(\bep) = \bX({\lasso}-\bbeta)$ as in (\ref{def-f}). 
Suppose Assumption \ref{assum:bX} holds and $\P\{\bv^\top\bep=c\}=0$ for all deterministic $\bv\in\R^n$ and real $c$.
Then almost surely
\bes
\nabla {\lasso} = \begin{pmatrix} (\bX_{\Shat}^\top \bX_{\Shat})^{-1}\bX_{\Shat}^\top 
\cr {\bf 0}_{\Shat^c \times n} \end{pmatrix}_{p\times n}
\ees
as well as
\bes
\nabla f(\bep) = \bP_{\Shat}\quad \hbox{ and }\quad 
\dv f(\bep) = \|\bP_{\Shat}\|_F^2 =  \trace \; \bP_{\Shat} = |\Shat|,
\ees
where $\Shat =\supp({\lasso})$ and $\bP_{\Shat}$ is as in (\ref{diff-Lasso}). 
\end{proposition}

\subsubsection{
{Variance formula and relative stability}}
Using \Cref{prop:variance-divergence} we obtain the following result whose proof
is given in \Cref{proof:bound-sparsity-general}.

\begin{theorem}
    \label{thm:variance-size-model-lasso}
    Consider the linear model (\ref{LM-lasso}) and ${\lasso}$ in (\ref{lasso}),
    with deterministic design $\bX$ satisfying \Cref{assum:bX}, true target vector $\bbeta$
    and noise $\bep\sim N({\bf 0}, \sigma^2 \bI_n)$.
    Then the variance of the size of the selected support satisfies
    \bel{th-model-size-1}
        \Var[|\Shat|]
    & = & {\E[|\Shat|] + \E\big[\|\bP_{\Shat} \bep\|^2/\sigma^2\big] 
    + \Var(g(\bep))-\sigma^2\E[\|\nabla g(\bep)\|^2]} 
    \cr &\le& \E[|\Shat|] + \E\big[\|\bP_{\Shat} \bep\|^2/\sigma^2\big], 
    \eel
    {where $g(\bep) = \bep^\top\bX({\lasso}-\bbeta)/\sigma^2$.}
    Consequently, $\Var[|\Shat|] \le 2n$ as well as 
    \bel{th-model-size-2}
    \Var[|\Shat|] 
     &\le& 
       3\E[|\Shat|]
    + 4\E\Big[|\Shat|\log\Big(\frac{ep}{{1\vee}|\Shat|}\Big)\Big]
    \\ \nonumber &\le & 3\E[|\Shat|]+4\E[|{\Shat}|]\log\Big(\frac{ep}{{1\vee}\E[|\Shat|]}\Big).
    \eel
\end{theorem}

A significant feature of the above theorem, and \Cref{prop:variance-divergence} as well, 
is the requirement of no condition on the true $\bbeta$, the penalty level $\lam$ or  
the design matrix $\bX$ beyond \Cref{assum:bX}. 
In particular, the restricted eigenvalue condition is not required. 

Using $a/b+b/(a\vee 1)-2\le(a-b)^2/(b(a\vee 1))$, 
an implication of \Cref{thm:variance-size-model-lasso} 
is the confidence interval
\bel{confidence-interval-Shat}
\qquad
\P\bigg(
\frac{|\Shat|}{\E[|\Shat|]}+\frac{\E[|\Shat|]}{|\Shat|\vee 1}-2 
\le C_\alpha \Big(\frac{3}{|\Shat|\vee 1}+\frac{4\log(ep)}{|\Shat|\vee 1}\Big)
\bigg)\ge 1-\alpha
\eel
for $\E[|\Shat|]$ with conservative $C_\alpha =1/\alpha$,
although $\E[|\Shat|]$ is 
not a conventional parameter due to its dependence on the specific choice of $\hbbeta$. 

{A sequence} of non-negative random variables $(Z_q)_{q\ge 1}$ is said to be relatively stable if $Z_q/\E[Z_q]$ converges to 1 in probability.
A direct consequence of \Cref{thm:variance-size-model-lasso} is that the model size $|\Shat|$ is relatively stable
provided that $\E[|\Shat|]$ is not pathologically small.
If the setting and assumptions of \Cref{thm:variance-size-model-lasso} are fulfilled
then
$$\E\Big[ \Big( \frac{|\Shat|}{\E|\Shat|} - 1 \Big)^2\Big]
\le \frac{3}{\E|\Shat|} + \frac{4\log(ep/\E|\Shat|)}{\E|\Shat|}
\le \frac{3 + 4\log(ep)}{\E|\Shat|}.
$$
Consequently, if one considers a sequence of regression problems such that $\E[|\Shat|]/\log(ep) \to +\infty$,
then $|\Shat|/\E|\Shat|$ converges to 1 in $L_2$ and in probability. 
It is known that $\P\{|\Shat| \ge k\}\approx 1$ when $\bX$ has iid $N({\bf 0},\bSigma)$ rows 
with $(\bSigma)_{jj}=1$, $\bbeta={\bf 0}$ and $\lam = \lam(k) = (\sigma/n^{1/2})\{L(k/p)-1\}$ 
under a mild side condition on $\bSigma$, 
where $L(t)$ is defined by $\P\{N(0,1)>L(t)\} = t$ 
\cite{sun2013sparse}[Proposition 14(ii)]. 
As $|\Shat|$ is expected to be larger when $\bbeta\neq {\bf 0}$, it would be reasonable 
to expect $\E[|\Shat|]\gg \log p$ when $\lam = \lam(k_n)$ with $k_n\gg \log p$. 

An informative benchmark to study the tightness of inequality 
    \eqref{th-model-size-2}
for the support of the Lasso is the case $\bX=\sqrt n \bI_p$ which reduces to the
Gaussian sequence model. Then $\lasso$ is the soft-thresholding operator
and $|\Shat|$ is the sum of $p$ iid Bernoulli
random variables with parameters $q_1,...,q_p\in(0,1)$ and 
$\Var[|\Shat|] = \sum_{j=1}^p q_j(1-q_j)$.
Under mild assumption on the probabilities $q_j$ (e.g., $q_j\le 1/2$ for all $j$),
the variance
$\Var[|\Shat|]$ is of the same order
as $\E |\Shat|$.
Hence the bound \eqref{th-model-size-2} is sharp up to a logarithmic factor.

\subsubsection{
{Linearity of the variance in the true model size}}
From (\ref{th-model-size-2}),
we can also obtain more explicit bounds on the variance of $|\Shat|$
by bounding from above $\E[|\Shat|]$. 
We provide below upper bounds on $\E[|\Shat|]$ under two assumptions on $\bX$:
the Sparse Riesz Condition (SRC) \cite{ZhangH08,zhang10-mc+}
and the Restricted Eigenvalue (RE) condition \cite{bickel2009simultaneous}.
Under both conditions, if the tuning parameter
 of the Lasso is large enough
then the squared risk $\|\bX({\lasso} - \bbeta)\|^2/n$
is bounded from above by $C(\bX)s_0\lambda^2$ 
with high probability \cite{ZhangH08,bickel2009simultaneous,zhang10-mc+}
and in expectation \cite{ye2010rate,bellec2016slope,bellec2016bounds},
where $C(\bX)$ is a multiplicative constant that depends on $\bX$.
We refer the reader to the books
\cite{buhlmann2011statistics,giraud2012high,hastie2015statistical}
and the references therein
for surveys of existing results.
Throughout the rest of this section,
denote by $s_0=\|\bbeta\|_0$ the number of nonzero coefficients, or sparsity, of 
the unknown coefficient vector $\bbeta$.

The Sparse Riesz Condition (SRC) \cite{ZhangH08,zhang10-mc+} on the design $\bX$ holds 
if for certain {$\eta\in (0,1)$ and nonnegative reals $\{\epsilon_1,\epsilon_2\}$ and integers $\{m,k\}$,} 
\begin{equation}
{|S| + \epsilon_2k \le 
\min_{B\subset[p]: |B \setminus S| \le m} \frac{2(1-\eta)^2m - \epsilon_1k}
{\phi_{\rm cond}\big(\bX_B^\top\bX_B\big)-1},}
\label{SRC}
\end{equation}
where $S$ denotes the support of $\bbeta$ 
{and} $\phi_{\rm cond}(\cdot)$ denotes the condition number.
Let $\phi(d) {\ge} \max\{\phi_{\rm cond}(\bX_B^\top\bX_B): |B|=d\}$
be an upper bound for the $d$-sparse condition number 
of the Gram matrix $\bX^\top\bX/n$. Given $d$ and {$\{\eta,\epsilon_1,\epsilon_2\}$}, 
the SRC can be viewed as a sparsity condition on $\bbeta$ as it 
holds with $k=\|\bbeta\|_0$ 
{when 
$\|\bbeta\|_0\le 2d (1-\eta)^2/\{(1+\epsilon_2)(\phi(d)-1)+2(1-\eta)^2+\epsilon_1\}$. 
In particular, the SRC holds under the RIP condition $\delta_{2s_0} < 1/2$ or $\delta_{3s_0} < 2/3$  
for sufficiently small $\{\eta,\epsilon_1,\epsilon_2\}$ 
where $\delta_{k}$ is the maximum spectrum norm of $\bX_B^\top\bX_B/n-\bI_B$ over $|B|\le k$.} 
When $\bX$ has iid $N({\bf 0},\bSigma)$
rows, we may take $d = a_1 {n /\log p}$
such that $\phi(d) = (1+a_2)\phi_{{\rm cond}}(\bSigma)$
is a valid upper bound for the $d$-sparse condition number of $\bX^\top\bX/n$ 
with probability $1-e^{-a_2n}$ for some small positive constants $a_1$ and $a_2$ \cite{ZhangH08}.  

Note that the original SRC in \cite{ZhangH08} is stated in terms
of ratio of maximal and minimal sparse eigenvalues instead
of sparse condition number as in \eqref{SRC}.
A common feature on the works on the SRC \cite{ZhangH08,zhang10-mc+}
is that $|\Shat|\lesssim s_0$ with large probability (up to constants depending
on the SRC constants and the tuning parameter).
We obtain $\E|\Shat| \lesssim s_0$ as a consequence,
provided that $\P(|\Shat|\lesssim s_0)$ is large enough.

\begin{proposition}
    \label{proposition:SRC-upperbound-expected-sparsity} 
    {Let $\bep\sim N({\bf 0}, \sigma^2 \bI_n)$ and $s_0=\|\bbeta\|_0$ in  
    the linear model (\ref{LM-lasso}). Let the tuning parameter of the Lasso \eqref{lasso} satisfy 
    \begin{equation}
        \label{def-lambda-SRC}
        \lambda\ge (\sigma/\eta)\sqrt{ (1+\tau)(2/n)\log(p/k)}\}
    \end{equation}
    with $\eta\in (0,1)$. 
    Assume that $\bX$ is deterministic with $\max_{j\in[p]}\|\bX\be_j\|^2/n \le 1$ and 
    that \eqref{SRC} holds with $\eps_1 = 8(1-\eta)/\{\sqrt{2\pi}t_0\}$, 
    $\eps_2 = 4/\{t_0^2(2\pi t_0^2+4)^{1/2}\}$ and some nonnegative integers $\{m,k\}$, 
    where $t_0=\eta^{-1}\sqrt{2\log(p/k)}$. 
     Then, $\P(|\Shat\setminus S| < m)\ge 1- (k/p)^\tau/(1+(\eta t_0)^2)$ 
    and $\E[|\Shat|]\le s_0+m + p (k/p)^\tau/(1+(\eta t_0)^2)$.
    Consequently, if $\tau\ge 1$ then $\E|\Shat| \le s_0+ m +k/(1+(\eta t_0)^2)$.} 
\end{proposition}

The proof is given in \Cref{sec:proof-upperbound-expected-sparsity-SRC}.
In particular if $\tau\ge 1$ and $s_0\approx m$ and $k\lesssim s_0$ then the SRC
grants $\E|\Shat| \lesssim s_0$.
The expectation $\E[|\Shat|]$ can also be bounded from above if 
both \Cref{assum:bX} and
the Restricted Eigenvalue (RE) condition \cite{bickel2009simultaneous} hold.


\begin{theorem}
    \label{thm:upperbound-expected-sparsity-RE}
    Consider the linear model (\ref{LM-lasso}) with $\bep\sim N({\bf 0}, \sigma^2 \bI_n)$ 
    and $s_0=\|\bbeta\|_0$.  
    Let $\tau,\gamma>0$, $\omega = \sigma(1+\tau)/\sqrt{n}$ and 
    $\lasso$ be as in (\ref{lasso}) with
    \begin{equation}
        \label{lambda-RE}
        \lambda=\sigma(1+\tau)(1+\gamma)\sqrt{(2/n)\log(ep/(s_0\vee 1))}.
    \end{equation}
    Assume that the columns of $\bX$ are normalized such that
    $\max_{j\in [p]}\|\bX\be_j\|^2\le n$. 
    Let $\Shat$ be the support of ${\lasso}$. 
    Then the Lasso satisfies
    $$
    \E\left[
    2\tau\bep^\top\bX({\lasso}-\bbeta)
    +
    \|\bX({\lasso} - \bbeta)\|^2
    \right]/n
    \le
    \frac{s_0(\lambda^2 + \omega^2)+(s_0\vee 1)\omega^2}{\RE^2(S,c_0)} + \frac{\omega^2}{2},
    $$
    where $\RE(S,c_0) = \inf_{\bu\in\R^p: \|\bu_{S^c}\|_1 \le c_0\sqrt{s_0\vee 1}\|\bu\|}
    (n^{-1/2}\|\bX\bu\| / \|\bu\|)$ and $S$ is the support of $\bbeta$, 
    provided that $c_0\ge \gamma^{-1}\sqrt{2(1+2\omega^2/\lam^2)}$, e.g. $c_0= 2/\gamma$. 
    If in addition \Cref{assum:bX} holds, then 
    \bel{bound-expceted-sparsity-RE} 
    && \E\Big[|\Shat| + \|\bX(\lasso - \bbeta)\|^2\big/(2\sigma^2\tau) \Big] 
    \\ \nonumber &\le& (\sqrt{\tau} +1/\sqrt{\tau})^2
    \left[
        \frac{(1+\gamma)^2\big\{s_0\log(ep/(s_0\vee 1)) + (s_0\vee 1)\big\}}
        { \RE^2(S,2/\gamma)} + \frac 1 4
    \right].
\eel
    \end{theorem}

The proof of \Cref{thm:upperbound-expected-sparsity-RE} is given
in \Cref{sec:proof-upperbound-expected-sparsity-RE}.
The Gaussian concentration theorem is used in
\cite{bellec2016bounds,bellec2016slope}
to obtain bounds on $\E[\|\bX({\lasso}-\bbeta)\|^2]$
as well as higher order moments of the squared risk;
similar arguments are used to derive \Cref{thm:upperbound-expected-sparsity-RE}.
If $\bX$ satisfies \Cref{assum:bX} then 
$\E\big[\bep^\top\bX({\lasso}-\bbeta)\big] = \E\big[|\Shat|\big]$, 
so that the argument leads to \eqref{bound-expceted-sparsity-RE}.  
Informally, this implies $\E[|\Shat|] \lesssim 1 + s_0\log(ep/(s_0\vee1))$ 
up to a multiplicative constant
that depends only on $\gamma,\tau$ and the restricted eigenvalue.
To our knowledge, this bound on the size of the model selected by the Lasso under the RE condition is new.
Previous upper bounds of the form $|\Shat|\lesssim s_0$ require
that both maximal and minimal sparse eigenvalues of $\bX^\top\bX/n$ are
bounded away from $0$ and $+\infty$,  
cf. \Cref{proposition:SRC-upperbound-expected-sparsity} above or 
\cite{ZhangH08}, \cite[Lemma 1]{zhang10-mc+}, \cite[(7.9)]{bickel2009simultaneous}, \cite[Theorem 3]{belloni2014} among others.
The major difference between such conditions and the RE condition is that the RE condition
does not require any bounds on the maximal sparse eigenvalues of $\bX^\top\bX/n$.
Inequality \eqref{bound-expceted-sparsity-RE} reveals that the RE condition is sufficient
to control $\E[|\Shat|]$ by $s_0$ times a logarithmic factor.
Under the RE condition, assumptions on the maximal sparse eigenvalues of $\bX^\top\bX/n$ are unnecessary to control $\E[|\Shat|]$.

The above bounds on $\E|\Shat|$ under the SRC or the RE condition yield
the following on the variance of $|\Shat|$ 
in virtue of \eqref{th-model-size-2}.
If \Cref{assum:bX} holds,

\begin{enumerate}
    \item If $\lambda$ is as in \eqref{lambda-RE} for some $\gamma,\tau>0$  and
$\max_{j=1,...,p}\|\bX\be_j\|^2 \le n$ then
$$\Var[|\Shat|] 
\le \frac{(s_0 \vee 1) C }{\RE^2(S,2/\gamma)} \log\Big(\frac{ep}{1\vee s_0}\Big)^2 . 
$$
where $C=C(\gamma,\tau)>0$ 
only depends on $\gamma,\tau$.

\item If $\bX$ satisfies the SRC \eqref{SRC}
    for some $\eta>0$ and {$0\le m,k \le p$,} 
    with $\lambda$ is as in \eqref{def-lambda-SRC} with {$\tau\ge 1$} then
    for some absolute constant $C'>0$,
    $$\Var[|\Shat|]
    \le
    C'\Big(
        {(s_0+m) \log\Big(\frac{ep}{s_0+m}}\Big)
    \Big).$$
\end{enumerate}
In other words, under the RE condition
or the SRC with {$m\approx s_0$} 
the standard deviation of the size of the model $|\Shat|$
is smaller than $\sqrt{s_0}$ up to logarithmic factors.
The bound is sharper under the SRC by a logarithmic factor.  


\subsubsection{Variance of degrees-of-freedom of penalized estimators}
Some techniques above are not specific to the Lasso.
For instance, for any estimator defined as the solution of a convex optimization
problem of the form given in
\Cref{remark:1-lipschitz-positive-symmetric-penalized-estimators},
the map
$f:\bep \to \bX(\hbbeta-\bbeta)$ is 1-Lipschitz and satisfies
\bel{eq-4-13}
    \E[(\sigma^2 \df - \bep^\top f(\bep))^2] 
         &\le& \sigma^2\E \|\bX(\hbbeta - \bbeta)\|^2 + \sigma^4 \E \|\nabla f(\bep)\|_F^2
    \cr  &\le& \sigma^2\E \|\bX(\hbbeta - \bbeta)\|^2 + \sigma^4 n
\eel
by \Cref{thm:second-order-stein}
where $\df = \dv f(\bep)$.
Similarly, by \Cref{prop:variance-divergence} we have
\bel{var-div-gen-betahat}
\Var[\df]& \le
    & \E \|\nabla f(\bep)\|_F^2 + \E \|\nabla f(\bep) \bep\|^2/\sigma^2
    \cr 
    & \le& 2n.
\eel

\subsection{SURE for SURE in high-dimensional linear regression}\label{sec:SOSforLassoEnet}
\ 
Again we consider linear regression 
with deterministic design $\bX\in\R^{n\times p}$.
With the notation of \Cref{sec:sure-of-sure},
consider the sequence model $\by = \bmu + \bep$ where $\bep\sim N({\bf 0}, \sigma^2\bI_n)$ 
and the unknown mean is
$\bmu=\bX\bbeta$, as in the linear model \eqref{LM-lasso}.

\subsubsection{Lasso}
\label{subsec:SURE-for_SURE-Lasso}
Set $\hbmu(\by)=\bX{\lasso}$ with 
the Lasso estimator \eqref{lasso}.
We have derived in the previous section the gradient of $\bep\to\bX{\lasso}$ almost everywhere
under \Cref{assum:bX}. It is instructive to use these calculations
to make explicit SURE for SURE from \Cref{sec:sure-of-sure} in the Lasso case.
Under \Cref{assum:bX}, 
$|\Shat| = \dv \hbmu = \trace((\nabla \hbmu)^2)$ by \Cref{prop-diff-lasso}, so that 
Stein's Unbiased Risk Estimate is
\bel{sure-lasso}
\SURE = \|\by - \bX{\lasso}\|^2 + \sigma^2(2|\Shat| -n)
\eel
as in \eqref{SURE-sigma}. Moreover, by 
\Cref{thm:sure-for-sure}, SURE for SURE in the Lasso case is 
\begin{align}
    {\Rhat_\tSURE} =  4\sigma^2 \|\by - \bX{\lasso}\|^2 +\sigma^4(4|\Shat| - 2n)
\label{SURE_for_SURE_Lasso}
\end{align}
which is an unbiased estimator of $R_\tSURE =\E[(\SURE - \|\bX({\lasso}-\bbeta)\|^2)^2]$. 
The identity $\E[\Rhat_\tSURE] = R_\tSURE$ for the Lasso appeared previously in \cite{dossal2013degrees}. 

    
    As $\nabla (\bX\lasso) =\bP_{\Shat}$ is a random projection, 
    \Cref{th-consistency} applies with $\Rhat'_{\tSURE}=\Rhat_{\tSURE}$, 
    so that SURE for SURE is consistent. 
    We explicitly state the consequences of this result in the following proposition. 

    \begin{proposition}
        \label{prop:consistency-SURE-for-SURE-Lasso}
    Consider the sequence model $\by = \bmu + \bep$
    where $\bep\sim N({\bf 0}, \sigma^2\bI_n)$ 
    and let $\bX\in\R^{n\times p}$ satisfy \Cref{assum:bX}.
    Let $\lasso$ be the Lasso \eqref{lasso}.
    Consider SURE in \eqref{sure-lasso}
    and SURE for SURE in \eqref{SURE_for_SURE_Lasso}.
    Then for any tuning parameter $\lambda\ge0$, the following holds:

\begin{enumerate}
\item {$\Var(\SURE) = \E\big[\big(\SURE-\E\big[\|\by-\bX\lasso\|^2\big]\big)^2\big] 
\le \E[\hat R_{\tSURE}] + \sigma^4n$.} 
\item
    $\E[\hat R_{\tSURE}] = 2\sigma^2\E[\|\by-\bX\lasso\|^2 + \SURE] \ge n\sigma^4$.
\item The self-bounding property
    $\Var[\hat R_{\tSURE}] \le 16\sigma^4\E[\hat R_{\tSURE}]$ holds.
\item
    Inequality
    $\E[
    |
    {\hat R_{\tSURE}}
    /
    {\E[\hat R_{\tSURE}]}
    -1
    |^2
    ]
    \le 16/n
    $ holds so that
    the ratio
    $
    \hat R_{\tSURE}
    /
    \E[\hat R_{\tSURE}]$ converges to 1 in $L_2$ and in probability
    as $n,p\to\infty$.
\item
    For any $\alpha,\gamma\in(0,1)$, with probability at least $1-\alpha-\gamma$,
    $$\Big|\SURE - \|\bX\lasso - \bmu\|^2\Big|^2
    \le {\gamma^{-1}\hat R_{\tSURE}\big(1-4(n\alpha)^{-1/2}\big)^{-1}.}
    $$
    \end{enumerate}
    \end{proposition}
    
    The proof is given in \Cref{sec:proof-consistency-SURE-lasso}.
    In the above, (i) provides in terms of SURE for SURE an upper bound 
    for the mean squared error of SURE when the prediction risk is viewed as 
    the estimation target of SURE, and (iv) is a consistency result for SURE for SURE
    in the Lasso case. The {non-asymptotic}  
    \Cref{prop:consistency-SURE-for-SURE-Lasso} holds with
    no restriction on {$(n,p)$ and} the tuning parameter $\lambda$. 
    The assumption-free nature of \Cref{prop:consistency-SURE-for-SURE-Lasso}
    is striking:
    SURE for SURE is consistent 
    even if the Lasso itself is not consistent for prediction in the sense that
    $\E[\|\bX\lasso-\bmu\|^2/(n\sigma^2)]$ is bounded away from 0.

    To bound the variance of $\hat R_{\tSURE}$,
    {\Cref{prop:consistency-SURE-for-SURE-Lasso}~(iv)} leverages
    the fact that for $\hbmu(\by)=\bX\lasso$,
    {$\nabla \hbmu(\by)$ is a random orthogonal projection} 
    and $\hat R_{\tSURE} = \hat R_{\tSURE}'$, cf. the discussion following
    \eqref{SURE4SURE-prime}.
    While $\nabla\hbmu(\by)$ for $\hbmu(\by)=\bX\hbbeta(\by)$
    is not a projection for other estimators $\hbbeta$
    such as the Elastic-Net studied in \Cref{subsection:Elastic-Net}, 
    the upper bounds in \Cref{th-consistency} still apply
    to $\hat R_{\tSURE}'$ for any convex regularized least-squares
    as explained in \Cref{remark:1-lipschitz-positive-symmetric-penalized-estimators}. 

    A drawback of the confidence region in {(v)} is the conservative constant {factor $\gamma^{-1}$.} 
    This can be fixed under common regularity assumptions made in sparse
    linear regression as follows with the approach of \Cref{sec:confidence-region-SURE}. 
Let $\{\tau,\gamma,\lam\}$ be as in \Cref{thm:upperbound-expected-sparsity-RE} and define 
$C_{\tau,\gamma} = \max(1,2\tau)(\sqrt{\tau}+1/\sqrt{\tau})^2\{4(1+\gamma)^2+5\}$. 
Under the conditions of \Cref{thm:upperbound-expected-sparsity-RE} including 
\Cref{assum:bX}, \eqref{bound-expceted-sparsity-RE} implies that  
the right-hand side of \eqref{th-2-2-1a} is bounded by $\sigma^4 2 n\eps_n^*$ 
with $\eps_n^* =C_{\tau,\gamma}  (s_0\vee 1)\{\log(p/(s_0\vee 1))\}/
\{n\RE^2(S,2/\gamma)\}$, so that 
\bes
   \P\Big\{\big|\SURE - \|\bmu - \hbmu\|^2\big| \le 1.96 \sigma^2\sqrt{2n}\Big\} \approx 95\% 
\ees
by \Cref{thm:sure-for-sure-tail} when $\eps_n^*=o(1)$, with $v_0^2=\eps_n=\sqrt{\eps_n^*}$, 
and similarly 
\bes
   \P\Big\{\|\bmu - \hbmu\|^2 \le \SURE + 1.645 \sigma^2\sqrt{2n}\Big\} \approx 95\%. 
\ees

\subsubsection{
{Two or more} Lasso estimators}
For the comparison of two Lasso estimators ${\lassoWith{\lam_1}}$ 
and ${\lassoWith{\lam_2}}$ with $\lam_1\neq\lam_2$, 
\bel{SURE-diff-lasso}
\SURE^{\tdiff} 
&=& \big\|\bX{\lassoWith{\lam_1}}-\by\big\|^2 - \big\|\bX{\lassoWith{\lam_2}}-\by\big\|^2 
\cr && {+} 2\sigma^2\big(\big|{\Shat^{(\lam_1)}}\big| - \big|{\Shat^{(\lam_2)}}\big|\big) 
\eel
provides $\E\big[ \SURE^{\tdiff} \big]
= \E\big[\big\|\bX({\lassoWith{\lam_1}}-\bbeta)\big\|^2 
- \big\|\bX({\lassoWith{\lam_2}}-\bbeta)\big\|^2\big]$, 
where $\Shat^{(\lam_j)}=\supp(\lassoWith{\lam_j})$. 
If $\bP_A$ is the projection onto the column space of $\bX_A$,
\bel{SUREforSURE-diff-Lasso}
\qquad&
    	\Rhat_{\tSURE}^{\tdiff} 
        =  4 \sigma^2 \big\|\bX\big({\lassoWith{\lam_1}} - {\lassoWith{\lam_2}}\big)\big\|^2 
        + 4 \sigma^4\trace\big( \big(\bP_{{\Shat^{(\lam_1)}}} - \bP_{{\Shat^{(\lam_2)}}}\big)^2\big)
\eel
provides 
$$\E \Rhat_{\tSURE}^{\tdiff} = R_{\tSURE}^{\tdiff} 
= \E\Big[\Big(\big\|\bX({\lassoWith{\lam_1}}-\bbeta)\big\|^2 
- \big\|\bX({\lassoWith{\lam_2}}-\bbeta)\big\|^2 - \SURE^{\tdiff}\Big)^2\Big].$$ 

{The results of \Cref{sec:SURE-oracle-inequality} are also} directly applicable
    to SURE-tuned Lasso estimators:
    If $\hat\lambda$ is the tuning parameter
    among $\{\lambda_1,...,\lambda_m\}$ with the smallest SURE,
    then for some absolute constant $C>0$
    $$\E\Big[n^{-1/2}\|\bX(\lassoWith{\hat\lambda}-\bbeta)\| - \min_{
        \lambda\in \{\lambda_1,...,\lambda_m\}
    }
    n^{-1/2}\|\bX(\lasso-\bbeta)\|
    \Big]\le C (m/n)^{1/4}
    .
    $$


\subsubsection{Elastic Net}
\label{subsection:Elastic-Net}
Similar computations can be carried out for other estimators such as the Group Lasso
or the Elastic Net. For instance, consider the Elastic Net estimator $\hbbeta_\tEN$
defined as the solution of the optimization problem
\begin{equation}
    \hbbeta_\tEN = \argmax_{\bb\in\R^p}
    \|\bX\bb - \by\|^2/2n + \lambda \|\bb\|_1
    + \gamma \|\bb\|^2/2,
    \label{elastic-net}
\end{equation}
where $\lambda,\gamma > 0$.
Set $\hbmu(\by) = \bX\hbbeta_\tEN$.
Then by similar arguments as in the Lasso case,
the KKT conditions of the optimization problem
\eqref{elastic-net} hold strictly almost everywhere in $\by$.
By differentiating the KKT conditions on a neighbourhood where
the KKT conditions hold strictly (the details are omitted),
the gradient of $\by\to \hbbeta_\tEN$ is given by
\bel{diff-beta-EN}
\nabla \hbbeta_\tEN = \begin{pmatrix} 
\Big(\gamma \bI_{\Shat} + \bX_{\Shat}^\top\bX_{\Shat}\Big)^{-1}\bX_{\Shat}^\top 
\cr {\bf 0}_{\Shat^c \times n} \end{pmatrix}_{p\times n}, 
\eel
and the gradient of $\by\to \bX\hbbeta_\tEN$ is given by
\bel{diff-f-EN}
\nabla(\bX\hbbeta_\tEN)
= \bX_{\Shat}(\gamma \bI_{\Shat} + \bX_{\Shat}^\top\bX_{\Shat})^{-1}\bX_{\Shat}^\top.
\eel
where $\Shat\subset[p]$ is the set of nonzero coefficients of $\hbbeta_\tEN$.
Stein's Unbiased Risk Estimate is given by
\begin{align}
\label{SURE-EN} 
 \SURE =
\|\by - \bX\hbbeta_\tEN\|^2 + 2 \sigma^2 \trace\left[
\bX_{\Shat}(\gamma \bI_{\Shat} + \bX_{\Shat}^\top\bX_{\Shat})^{-1}\bX_{\Shat}^\top
\right] - \sigma^2n,
\end{align}
and SURE for SURE in the Elastic Net case is
\begin{align}
    {\Rhat_\tSURE} = 4\sigma^2\|\bX\hbbeta_\tEN - \by\|^2 + 4\sigma^4\| \bX_{\Shat}(\gamma \bI_{\Shat} + \bX_{\Shat}^\top\bX_{\Shat})^{-1}\bX_{\Shat}^\top\|_F^2
- {2\sigma^4n.}  
\end{align}
By \Cref{thm:sure-for-sure},
this is an unbiased estimate of $\E[(\SURE - \|\bX(\bbeta-\hbbeta_\tEN)\|^2)^2]$. 
SURE for SURE $\Rhat_{\tSURE}^{\tdiff}$ for the difference between two E-nets or between the Lasso and E-net can be derived similarly as in (\ref{SUREforSURE-diff-Lasso}). We omit the details. 

\begin{remark}
    Let $\df=
    \trace[
\bX_{\Shat}(\gamma \bI_{\Shat} + \bX_{\Shat}^\top\bX_{\Shat})^{-1}\bX_{\Shat}^\top
]
    $.
    Since $\df$ is the divergence of the function
    $\bep \to \bX(\hbbeta_{\tEN}-\bbeta)$,
    \Cref{prop:variance-divergence} implies that
    \bes
    \Var[\df]
    &\le&
    \E[ \|
        \bX_{\Shat}(\gamma \bI_{\Shat} + \bX_{\Shat}^\top\bX_{\Shat})^{-1}\bX_{\Shat}^\top
        \|_F^2
    ]
    \cr && {+}
    \E[
    \|
        \bX_{\Shat}(\gamma \bI_{\Shat} + \bX_{\Shat}^\top\bX_{\Shat})^{-1}\bX_{\Shat}^\top
    \bep
    \|^2]/\sigma^2.
    \ees
    If $\bP_{\Shat}$ is the orthogonal projection onto the span of the columns
    of $\bX_{\Shat}$ then the second term satisfies
    $\E[
    \|
        \bX_{\Shat}(\gamma \bI_{\Shat} + \bX_{\Shat}^\top\bX_{\Shat})^{-1}\bX_{\Shat}^\top
    \bep
    \|^2]/\sigma^2
    \le \E[\|\bP_{\Shat} \bep\|^2]/ \sigma^2
    $.
    Since {the right-hand of \eqref{th-model-size-1} is no greater than that of \eqref{th-model-size-2}} 
    for any random $\Shat$ {by the proof of \Cref{thm:variance-size-model-lasso}},
    we obtain
    \bes
    \Var[\df]
    &\le&
    \E[ \|
        \bX_{\Shat}(\gamma \bI_{\Shat} + \bX_{\Shat}^\top\bX_{\Shat})^{-1}\bX_{\Shat}^\top
        \|_F^2
    ]
    \cr && {+}  \E[ 2|\Shat|
    +  4|\Shat|\log(ep \big/ \{1\vee |\Shat|\})]
    \\&\le&
    3 \E[|\Shat|]
    +  4\E[|\Shat|\log(ep \big/ \{1\vee |\Shat|\})].
    \ees
\end{remark}

\subsection{De-biasing nonlinear estimators in {linear regression}}
\label{sec:debiasing}
Consider a linear regression model 
\bel{LM}
\by = \bX\bbeta + \bep
\eel
with an unknown target vector $\bbeta\in \R^p$, a Gaussian noise vector 
$\bep\sim N({\bf 0},\sigma^2\bI_n)$, and a Gaussian design matrix $\bX\in \R^{n\times p}$ with iid 
$N({\bf 0},\bSigma)$ rows. We assume that the covariance matrix $\bSigma$ is known and invertible.

This section explains how to construct an estimate of a linear contrast 
\bel{theta}
\theta  = \big\langle \ba_0, \bbeta \big\rangle 
\eel
from an initial estimator $\hbbeta$.
Here and in the sequel,
$\langle\cdot,\cdot\rangle$ denotes the scalar product in $\R^n$.
Define
\bel{u_0}\quad &
\bu_0 = \bSigma^{-1}\ba_0/\big\langle\ba_0,\bSigma^{-1}\ba_0\big\rangle,
\quad
\bz_0 = \bX\bu_0, 
\quad
\bQ_0 = \bI_{p\times p} - \bu_0\ba_0^\top
\eel
and assume for simplicity that $\ba_0$ is normalized such that 
$$\langle \ba_0, \bSigma^{-1} \ba_0 \rangle = 1.$$
By definition of $\bu_0$, $\bz_0 \sim N({\bf 0},\bI_n)$ and $\bz_0$ is independent of $\bX \bQ_0$.

We assume throughout this section that we are given
an initial estimator $\hbbeta$.
Since $\bX =\bz_0\ba_0^\top + \bX\bQ_0$ and the two random
vectors $\bz_0,\bX\bQ_0$ are independent,
we view $\hbbeta$ as a function with three arguments
$\hbbeta=\hbbeta(\by,\bz_0,\bX\bQ_0)$
and we assume that the partial derivatives 
$(\partial/\partial \by)\hbbeta$ and $(\partial/\partial \bz_0)\hbbeta$
exist almost everywhere.

The estimator $\hbbeta$ provides an initial estimate
of the unknown parameter $\theta$ \eqref{theta}
by the plug-in $\langle \ba_0, \hbbeta\rangle$.
However, this estimator may be biased, and a first attempt
to fix the bias is the following one-step MLE correction
in the direction given by the one dimensional model $\{\hbbeta + t\bu_0, t\in\R\}$,
\begin{equation}
    \langle \ba_0, \hbbeta\rangle + \frac{\langle \bz_0, \by -
    \bX\hbbeta\rangle}{\|\bz_0\|^2}.
    \label{LDPE}
\end{equation}
Variants of the above de-biasing scheme have been considered in \cite{zhang2011statistical, 
ZhangSteph14,BelloniCH14,Buhlmann13,GeerBR14,JavanmardM14a,javanmard2018debiasing}, 
among others. 
We multiply by $\|\bz_0\|^2$ to avoid random denominators;
the random variables $\|\bz_0\|^2$ is chi-square with $n$ degrees
of freedom, equal to $n+O(\sqrt n)$ with overwhelming probability
so that $\|\bz_0\|^2\approx n$ describes the number of observations.

When constructing the estimator \eqref{LDPE}
above by the one-step MLE correction, the statistician hopes
that the quantity
\begin{equation}
    {\|\bz_0\|^2 \langle \ba_0, \hbbeta-\bbeta\rangle + 
    \langle \bz_0, \by - \bX\hbbeta\rangle}
\end{equation}
is asymptotically standard normal; this is the ideal result
to construct confidence intervals for the unknown parameter \eqref{theta}
at the $\sqrt n$-adjusted rate.

By simple algebra we have
\begin{equation}
    \label{de-biased-quantity}
    {\|\bz_0\|^2 \langle \ba_0, \hbbeta - \bbeta\rangle + 
    \langle \bz_0, \by - \bX\hbbeta\rangle}
    = 
    {\bz_0^\top \bep}
    -
    {\bz_0^\top \bX{\bQ_0(\hbbeta-\bbeta)}}
    .
\end{equation}
The random variable $ \bz_0^\top \bep$ in the right-hand side
is mean-zero and $\bz_0^\top \bep/(\sigma/\sqrt n)$ 
is asymptotically standard normal.
It remains to understand 
the bias {term} ${\bz_0^\top \bX{\bQ_0(\hbbeta-\bbeta)}}$.
For the derivation below, we will argue conditionally on $(\bep,\bX\bQ_0)$
and define $f:\R^n\to\R^n$ by
$$f(\bz_0) = \bX{\bQ_0(\hbbeta-\bbeta)}.$$
The quantity $f(\bz_0)$ is still biased
and Stein's formula lets us quantify the remaining bias in \eqref{de-biased-quantity}
exactly as follows
$$
\E\left[{\bz_0^\top \bX{\bQ_0(\hbbeta-\bbeta)}}\Big|\bX\bQ_0, \bep\right]
=
\E\left[{\bz_0^\top f(\bz_0)}\Big|\bX\bQ_0, \bep\right]
=
\E\left[{\dv f(\bz_0)}\Big| \bX\bQ_0, \bep\right]
.
$$
The partial derivatives $(\partial/\partial z_{0i})f_i$ 
where $f_i$ is the $i$-th coordinate of $f$
can be computed by the chain rule
$$
\frac{\partial f_i}{\partial z_{0i}}
= \be_i^\top \bX\bQ_0\Big[
\langle \ba_0,\bbeta\rangle 
\frac{\partial \hbbeta}{\partial y_i}
+
\frac{\partial \hbbeta}{\partial z_{0i} }
\Big]
.
$$
Hence, the divergence of $f$, which quantifies the remaining bias in 
\eqref{de-biased-quantity} is
$\dv f
= \langle \ba_0, \bbeta\rangle \hat{\nu}  + \hat B
$, where
\begin{equation}
\hat{\nu}= \trace \Big[ \bX\bQ_0 \frac{\partial \hbbeta}{\partial \by}
\Big]
,
\quad
\hat B = \trace \Big[ \bX\bQ_0 \frac{\partial \hbbeta}{\partial \bz_0}
\Big]
.
\label{def-hat-df-and-B}
\end{equation}
It will be convenient to write $\dv f$ instead as
\begin{equation}
\dv f = \langle \ba_0, \bbeta-\hbbeta \rangle \hat{\nu} + \hat A
\qquad
\text{ where }
\hat A = \hat B + \langle \ba_0, \hbbeta\rangle \hat{\nu}.
\label{def-hat-A}
\end{equation}
The quantities $\hat{\nu}, \hat A$ and $\hat B$ above can be constructed from
the observed data since they only depend on $\bX,\bQ_0,\by$
and the derivatives of $\hbbeta$.
However, the quantity $\langle \ba_0, \bbeta\rangle$ is unknown;
it is the parameter of interest that we wish to estimate.
This motivates the estimator of $\theta=\langle \ba_0,\bbeta\rangle$ defined by
\begin{equation}
\hat\theta
= \langle \ba_0, \hbbeta\rangle 
+\frac{\bz_0^\top(\by - \bX\hbbeta) + \hat A}{\|\bz_0\|^2-\hat{\nu}}
\label{def-theta}
\end{equation}
with $\hat A$ and $\hat{\nu}$ as in \eqref{def-hat-df-and-B} and \eqref{def-hat-A}.
This estimator $\hat\theta$ is constructed so that the random variable
\bel{pf-th-debias}
 {(\|\bz_0\|^2 - \hat{\nu})(\hat\theta - \theta)}
- {\bz_0^\top \bep}
 &=& {\hat A + \hat{\nu}\langle \ba_0,\bbeta-\hbbeta\rangle - \bz_0^\top f(\bz_0)}
\cr &=& {\dv f(\bz_0) - \bz_0^\top f(\bz_0)}
\eel
is exactly mean-zero by the first-order Stein's formula \eqref{stein-first-order-formula}.
Furthermore, the variance of this random variable can be expressed exactly
in terms of the derivatives of $f$ thanks to the Second Order Stein formula \eqref{eq1}.
Similarly, the above equality can be rewritten as
\bel{mean-zero-with-bep-outside}
(\|\bz_0\|^2 - \hat{\nu})(\hat\theta - \theta)
= 
\dv f(\bz_0) - \bz_0^\top (f(\bz_0) - \bep),
\eel
which is equal to $\dv g(\bz_0) - \bz_0^\top g(\bz_0)$
for $g(\bx) = f(\bx)-\bep$ since $f$ and $g$ have the same divergence.
Hence the random variable \eqref{mean-zero-with-bep-outside} is exactly mean zero
by the first-order Stein's formula, and the Second Order Stein formula \eqref{eq1}
provides an exact identity for its variance.
We gather the above derivation in the following theorem.

\begin{theorem}
    \label{thm:debiasing-general}
    Let $\hbbeta$ be an estimator such that, if we write it as a function
    $\hbbeta(\by,\bz_0,\bX\bQ_0)$, all partial derivatives of the function 
    $\bX\bQ_0\hbbeta$ with respect to $\by$ and $\bz_0$
    exist and are in $L_2$.
    Define the estimator $\hat\theta$ of $\theta=\langle \ba_0,\bbeta\rangle$
    by \eqref{def-theta}, with $\hat{\nu}$ and $\hat A$ as in
    \eqref{def-hat-df-and-B} and \eqref{def-hat-A}.
    Then the random variable
    \bel{adjusted-random-variable} 
        \frac{(\|\bz_0\|^2 - \hat{\nu})(\hat\theta - \theta)}{\sigma\sqrt n}
        - \frac{\bz_0^\top \bep}{\sigma\sqrt n}
    \eel
    is exactly mean-zero and its variance is exactly equal to
    \bel{asymptotic-variance-negligible} 
        \frac{1}{n\sigma^2}
        \left(
            \E\left[\|\bX{\bQ_0(\hbbeta-\bbeta)}\|^2\right] 
            +
            \E\left[\trace((\nabla f(\bz_0))^2)\right]
        \right)
     \eel
    where $f(\bz_0) = \bX{\bQ_0(\hbbeta-\bbeta)}$.
    Furthermore, 
    the random variable
    $(\|\bz_0\|^2 - \hat{\nu})(\hat\theta - \theta)$
    is also mean-zero with variance equal to
    \bel{asymptotic-variance-including-bep} 
            \E\left[\|\bX\hbbeta - \by- \bz_0\ba_0^\top(\hbbeta-\bbeta)\|^2 
            +
            \trace((\nabla f(\bz_0))^2)\right]
            .
     \eel
\end{theorem}

\Cref{thm:debiasing-general} is a direct consequence of 
the Second Order Stein formula \eqref{eq1} 
and the analysis summarized in (\ref{pf-th-debias}).

The random variable 
{$\bz_0^\top \bep/(\sigma\sqrt n)$} is asymptotically standard normal. 
The first claim
of the above result implies that
$({(\|\bz_0\|^2 - \hat{\nu})(\hat\theta - \theta)})/{\sigma\sqrt n}$ is also asymptotically
normal if \eqref{asymptotic-variance-negligible} converges to 0 as $n\to+\infty$. This provides a general strategy
to derive asymptotic normality results;
however the calculation of the gradients
of $\hbbeta$ and $f$
has to be carried out on case-by-case basis which is outside of the scope of
the present paper.



The above construction provides a general scheme to de-bias
an initial estimator $\hbbeta$ for the estimation of a linear
contrast $\theta=\langle \ba_0,\bbeta\rangle$
when the covariance matrix $\bSigma$ is known 
    because $(\|\bz_0\|^2-\hat\nu)(\hat\theta-\theta)$ is mean-zero and its variance
    is exactly given by \eqref{asymptotic-variance-including-bep}. Notably,
    both $(\|\bz_0\|^2-\hat\nu)(\hat\theta-\theta)$ and the quantity inside
    the expectation in \eqref{asymptotic-variance-including-bep} only depends on the 
    unknown parameter of interest $\theta=\ba_0^\top\bbeta$
    and known observable quantities
    $\ba_0^\top\hbbeta$, $\hat A$,
    $\hat \nu$, $\hbbeta$ and its derivatives. One may consider the z-score
    $(\|\bz_0\|^2-\hat\nu)(\hat\theta-\theta)/V^*(\theta)^{1/2}$ where $V^*(\theta)$ is
    the quantity inside the expectation in \eqref{asymptotic-variance-including-bep}.
    First and Second Order Stein formulae provide this z-score as the starting point
    to de-bias the general estimator $\hbbeta$ and normalize its variance,
    as soon as the partial derivatives of $\hbbeta$ exist.
    Studying the asymptotic distribution of this z-score requires tools beyond
    the scope of the present work and will be the subject of a forthcoming paper.

A notable feature of the above result is the random variable $\hat \nu$
whose role is to adjust multiplicatively the random variable $(\hat\theta-\theta)$
so that $(\|\bz_0\| - \hat \nu)(\hat\theta - \theta)$ is exactly mean-zero.
This adjustment accounts for the degrees-of-freedom of the initial
estimator $\hbbeta$.
We refer to our concurrent paper \cite{bellec_zhang_future_lasso-dof-adjustment}
for theory of degrees-of-freedom adjustment in semi-parametric inference about
a preconceived one-dimensional parameter $\theta=\langle \ba_0, \bbeta\rangle$.

\subsection{{Monte Carlo approximation of divergence} 
}
\label{sec:MonteCarlo}
The Second Order Stein formula and the techniques 
presented in this paper also suggest a Monte Carlo
method to approximate the divergence in the general case. 

Suppose we are interested in the approximation
of $\dv f(\by)$ at the currently observed vector $\by$. 
Assume that the function $f(\cdot)$ is 1-Lipschitz 
and its value can be quickly computed for small perturbations of $\by$,
say, $f(\by + a \bz)$ for small $a\bz$. 
{For example, when} $f(\by) = \bX\hbbeta$ in the linear model with a convex regularized least-squares 
estimator $\hbbeta$, the 1-Lipschitz condition holds automatically \cite{bellec2016bounds} 
{as discussed in \Cref{remark:1-lipschitz-positive-symmetric-penalized-estimators},} 
and if $\hbbeta(\by)$ has already been computed by an iterative algorithm, 
the computation of $\hbbeta(\by+a\bz)$ would typically be fast as 
one can use $\hbbeta(\by)$ as a starting point (``warm start'') to compute $\hbbeta(\by+a\bz)$. 
Next, with the 1-Lipschitz function
$h(\bz) = a^{-1}(f(\by+a\bz) - f(\by))$
and $\bz\sim N({\bf 0},\bI_n)$ independent of $\by$,
if $\E_{\bz}$ denotes the expectation with respect to $\bz$
conditionally on $\by$,
we have by the Gaussian Poincar\'e inequality that
$$
\E_{\bz}
[(\bz^\top h(\bz)- D_0 )^2]
\le
\E_{\bz}[
\|h(\bz) + \nabla h(\bz) \bz\|^2
]
\le
4n
$$
with $D_0= \E_{\bz} \dv h(\bz) = \int_{\R^n} (2\pi)^{-n/2}e^{-\|\bx\|^2/2} \dv f(\by + a\bx) d\bx$.
Hence
If we compute {$f(\by+a\bz_j)$} at $m$ independent Gaussian perturbations
$\bz_1,...,\bz_m\sim N(0,\bI_n)$, inequality 
$$\E\Big[\Big(\frac 1 m \sum_{j=1}^m \bz_j^\top h(\bz_j) - D_0\Big)^2{\Big|\by}\Big]\le \frac{4n}{m}$$
holds.
Here the function $\by\to \dv(\by)$
is locally integrable and almost surely
bounded by $n$ thanks to the Lipschitzness of $f$.
For almost every $\by$,
$\by$ is a Lebesgues point of 
$\dv f$
so that $D_0 \to \dv (\by)$
as $a\to 0$ by the
Lebesgues differentiation theorem.
Hence $D_0\approx \dv f(\by)$
for small enough $a>0$
and $\frac 1 m \sum_{j=1}^m \bz_j^\top h(\bz_j)$ provides a useful
approximation of the divergence
thanks to {the above conditional variance bound.} 
For large $m$, more precise results can be obtained by the central limit theorem.
Note that by the Second Order Stein formula, $\frac 1 m \sum_{j=1}^m \bz_j^\top h(\bz_j)$
is also close to the empirical average $\bar D^{(m)}=
\frac 1 m \sum_{j=1}^m \dv f(\by+a\bz_j)$
thanks to
$$\E\Big[\Big(\frac 1 m \sum_{j=1}^m \bz_j^\top h(\bz_j) - \bar D^{(m)}\Big)^2\Big]
= \frac 1 m \E\left[\|h(\bz_1)\|^2 + \trace\{[\nabla h(\bz_1)]^2\}\right]
\le \frac{2n}{m}
.$$

We apply this probabilistic procedure to the Elastic-Net
and Singular Value Thresholding (SVT) \cite{candes2013unbiased}
for which explicit formulae for $\hat{\text{df}}=\trace[\nabla f(\by) ]$
are available. 

{Indeed,} 
\cite[Equation (1.8)]{candes2013unbiased}
provides an explicit formula for the degrees-of-freedom
of SVT with tuning parameter $\lambda$:
If $\hat\bB$ soft-thresholds the singular vlaues of an observed
matrix $\bY\in\R^{m\times n}$ with tuning parameter $\lambda$,
the divergence of $\hat\bB$ with respect to $\bY$ is given by
$$
\df
=\sum_{i=1}^{q\wedge n}
\left\{
I_{\{\sigma_i>\lambda\}} + |q-n|(1-\lambda/\sigma_i)_+
\right\}
+ 2 \sum_{i=1}^{q\wedge n}
\sum_{j=1,j\ne i}^{q\wedge n}
\frac{\sigma_i(\sigma_i - \lambda)_+}{\sigma_i^2 - \sigma_j^2}
$$
where $\sigma_1,...,\sigma_{n\wedge q}$ are the singular values of $\bY$.
We then compare on simulated data this exact formula to the
above random approximation scheme.
For $n=100, q=101$, $\lambda=10.0$, with $\bY$ being the sum of standard
normal noise plus a ground-truth rank-10 matrix,
we apply the above algorithm with $m$ perturbations
$(\bY+a\bZ_j)_{j=1,...,m}$ for various values of $m$  and compute
$\df_{approx} = \frac 1 m \sum_{j=1}^m \trace\{\bZ_j^\top h(\bZ_j) \}$
where
$h(\bZ_j) = a^{-1}(\hat \bB(\bY+a\bZ_j) - \hat\bB(\bY))$
with $a=0.0001$
as explained above. The results are in \Cref{table:SVT}.  

In the case of the Elastic-Net with $\ell_1$ parameter $\lambda>0$
and $\ell_2$ parameter $\gamma>0$, we draw a similar experiment
with the exact formula for degrees-of-freedom being given
by $\df = \trace[ \bX_{\Shat}( \bX_\Shat^\top \bX_\Shat + n \gamma)^{-1}\bX_{\Shat}^\top]$ 
{as in \Cref{subsection:Elastic-Net}.}  
With
$n = 500$, $p = 400$, and again $a = 0.001$,
$\bX$ having independent symmetric $\pm1$, $\lambda = 0.8\sqrt{4\log(p)/n}$,
$\gamma=0.2\sqrt{4\log(p)/n}$,
we obtain
the standard errors and boxplots in \Cref{table:enet}.

\begin{sidewaysfigure}[ht]
\centering
\begin{tabular}{lrrrrrrrrrrr}
\toprule
{$m$} &  10 &  25 &  50 &  75 &  100 &  125 &  150 &  175 &  200 &  225 &  df\_exact \\
\midrule
mean &        2600.98 &        2600.17 &        2598.41 &        2597.98 &        2597.39 &        2597.74 &        2597.91 &        2597.78 &        2597.90 &        2597.73 &   2598.75 \\
std  &          22.06 &          11.96 &           9.52 &           8.21 &           6.91 &           5.98 &           5.41 &           4.72 &           5.02 &           5.56 &      NA \\
\bottomrule
\end{tabular}
\includegraphics[width=6in]{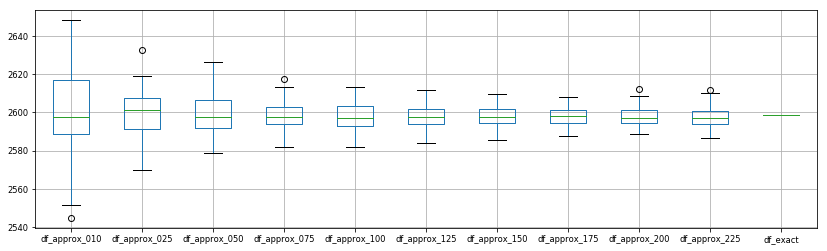}
\caption{Approximate $\df$ for SVT, computed over 50 realisations
of $(\bZ_1,...,\bZ_m)$ for various values of $m$.
The rightmost column is the excact formula from \cite{candes2013unbiased}.
The value of $\bY$ is the same over all 50 realisations.
\label{table:SVT}
}
\label{fig:SVT}
\end{sidewaysfigure}

\begin{sidewaysfigure}[ht]
\centering
\label{fig:enet}
\begin{tabular}{lrrrrrrrrrrrr}
\toprule
$m$ &  10 &  25 &  50 &  75 &  100 &  125 &  150 &  175 &  200 &  225 &  250 &  df\_exact \\
\midrule
mean &          52.16 &          52.06 &          52.23 &          52.39 &          52.34 &          52.33 &          52.40 &          52.35 &          52.41 &          52.40 &          52.43 &     52.45 \\
std  &           2.91 &           1.77 &           1.36 &           1.18 &           1.02 &           0.88 &           0.89 &           0.88 &           0.77 &           0.72 &           0.68 &      NA \\
\bottomrule
\end{tabular}
\includegraphics[width=6in]{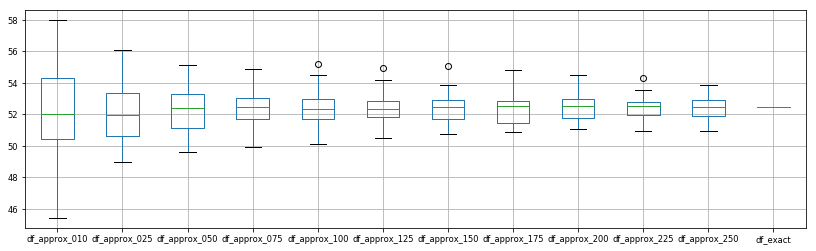}
\caption{Approximate $\dv (\bX\hbbeta)$ for the Elastic-net, computed over 50 realisations
of $(\bz_1,...,\bz_m)$ for various values of $m$
The rightmost column is the excact formula from \cite{candes2013unbiased}.
Corresponding boxplots are visible in \Cref{fig:SVT}.
The value of $(\bX,\by)$ is the same over all 50 realisations.
\label{table:enet}
}
\end{sidewaysfigure}

The experiments show that the above approximation scheme provides 
good approximations
in these special cases where exact formula are available.
Hence it could also be useful for estimators where no exact
formula is available for the divergence.

\appendix
\section{Non-smooth functions}
\label{appendix:proof(ii)-(iii)}

\begin{proof}[Proof of \Cref{thm:second-order-stein} (ii) for Lipschitz functions]
    If $f$ is Lipschitz, then each component $f_i$ of $f$ is also Lipschitz.
        Hence $f_i$ belongs to the space $W^{1,2}(\gamma_n)$ defined above
        \eqref{norm-sobolev} and the weak gradient of $f_i$ is equal almost
        everywhere to its gradient in the sense of Frechet differentiability
        (cf., e.g.  \cite[Theorem 4-6  pp 279-281]{evans1997partial}).
        Thus (ii) is a consequence of (iii).
\end{proof}

\begin{proof}[Proof of \Cref{thm:second-order-stein} (iii) for $f_i\in W^{1,2}(\gamma_n)$] 
    Since $W^{1,2}(\gamma_n)$ is the completion
        with respect to the norm \eqref{norm-sobolev}
        of the space $C^\infty_0(\R^n)$
    of smooth functions
    with compact support,
    for each coordinate $i=1,...,n$ there exists a sequence $(g_{i,q})_{q\ge 1}$
    of $C^\infty_0(\R^n)$ functions with
    $\max_{i=1,...,n}\E[(f_i-g_{i,q})^2 + \|\nabla f_i - \nabla g_{i,q}\|^2]\to 0$
    as $q\to+\infty$.
    Define $g_q:\R^n\to\R^n$ as the function with components $g_{1,q},...,g_{n,q}$.
    By considering a subsequence, we may assume that for all $q\ge 1$,
\bes
\E\Big[ \| g_q(\bz) - f(\bz) \|^2 \Big] + \E\Big[ \| \nabla g_q(\bz) - \nabla f(\bz)\|_F^2 \Big]\le 2^{-q-2}
\ees 
which implies that $g_q\to f$ and $\nabla g_q\to \nabla f$ pointwise
almost surely by the Borel-Cantelli lemma.
Let $X_q = \bz^\top g_q(\bz) - \dv g_q(\bz)$ and $X  = \bz^\top f(\bz) - \dv f(\bz)$. 
Then $X_q\to X$ almost surely. The triangle inequality and \Cref{thm:second-order-stein} (i) applied to $g_k - g_{k+1}$ yields 
\bes
&& \big\{\E\big[\big(X_q-X\big)^2\big]\big\}^{1/2}
\cr &\le& \sum_{k=q}^\infty \big\{\E\big[\big(X_k-X_{k+1}\big)^2\big]\big\}^{1/2}
\cr &\le& \sum_{k=q}^\infty \big\{\E\big[ \| g_k(\bz) - g_{k+1}(\bz) \|^2 \big] + 
\E\big[ \| \nabla g_k(\bz) - \nabla g_{k+1}(\bz)\|_F^2 \big]\big\}^{1/2}
\cr &\le& \sum_{k=q}^\infty 2^{-k/2} \to 0. 
\ees
Hence, with another application of \Cref{thm:second-order-stein}(i), 
\bes
&& \E\Big[\Big(\bz^\top f(\bz) - \dv f(\bz)\Big)^2\Big] 
\cr &=& \lim_{q\to\infty} \E\Big[\Big(\bz^\top g_q(\bz) - \dv g_q(\bz)\Big)^2\Big]
\cr &=& \lim_{q\to\infty} \E\Big[\|g_q(\bz)\|_2^2 + \trace\big((\nabla g_q)^2(\bz)\big)\Big]
\cr &=& \E\Big[\|f(\bz)\|_2^2 + \trace\big((\nabla f)^2(\bz)\big)\Big].
\ees
For the first and last equality, we use the fact that
if two sequences $(Z_q)_{q\ge 1}$ and $(Y_q)_{q\ge1}$ and two random variables
$Y_\infty,Z_\infty$
are such that $\E[(Y_q-Y_\infty)^2]\to 0$ and $\E[(Z_q-Z_\infty)^2] \to 0$ as $q\to+\infty$
then $\E[Y_q^2]\to \E[Y_\infty^2]$ and $\E[Y_q Z_q] \to \E[ Y_\infty Z_\infty]$.
This completes the proof. 
\end{proof}

\section{Proof of consistency} 
\label{sec:appendix-proof-risk-of-sqrt-SURE}
\begin{proof}[Proof of \Cref{th-consistency}] 
        (i) Since $\E[\SURE]=\E[\|\hbmu-\bmu\|^2]$,
    \bes
    \E\big[\Rhat'_{\tSURE}\big] 
    = \E\big[2\sigma^2\|\by - \hbmu\|^2 + 2\sigma^2\|\hbmu-\bmu\|^2\big]
    \ge \sigma^2\E\big[\|\bep\|^2\big] = \sigma^4n.
    \ees
    
    (ii) As $((\SURE)_+^{1/2} - \|\hbmu - \bmu\|)^2\le |\SURE - \|\hbmu - \bmu\|^2|$, by
    the triangle inequality
    \bes
        && \E\Big[\Big({(\SURE)_+}^{1/2} - \E[ \|\hbmu - \bmu\|^2]^{1/2}\Big)^4\Big]^{1/4}
        \cr && \le
        R_{\tSURE}^{1/4}
        + \E\Big[\Big(\|\hbmu - \bmu\| - \E[\|\hbmu - \bmu\|^2]^{1/2}\Big)^4\Big]^{1/4}.
    \ees
    If $\by\to\hbmu$ is a 1-Lipschitz function then
    $\bep\to \|\hbmu - \bmu\|$ is also 1-Lipschitz and
    the second term above is bounded from above as follows:
    \bes
      &&\E[(\|\hbmu - \bmu\| - \E[\|\hbmu - \bmu\|^2]^{1/2})^4]^{1/4}
      \\&\le&
        \E[(\|\hbmu - \bmu\| - \E \|\hbmu - \bmu\|)^4]^{1/4}
        +
        \E[(\E\|\hbmu - \bmu\| - \E[\|\hbmu - \bmu\|^2]^{1/2})^4]^{1/4}
      \\&\le&  (2+1)\sigma = 3\sigma
    \ees
    by {$\int_0^\infty \P(|(\|\hbmu - \bmu\| - \E \|\hbmu - \bmu\| |> \sigma x) {4x^3} dx \le \int_0^\infty 
    {8} e^{-x^2/2} {x^3} dx = 16$} 
    for the first term,
    and by
    $0\le\E[\|\hbmu - \bmu\|^2] - (\E\|\hbmu - \bmu\|)^2 \le \sigma^2$
    thanks to the Gaussian Poincar\'e inequality for the second term. 
    {This yields \eqref{QUARTIC-RISK-SQRT-SURE}. 

    (iii) Let $\bh = \hbmu-\bmu$ and $\df = \dv \hbmu$.  
    By \eqref {SURE-sigma}, $\SURE$ is a variable of form \eqref{eq:general-random-variable} with 
    $f(\bep) = \bep -2\bh$ and $g(\bep) = - \|\bh\|^2$. 
    As $\nabla f(\bep) = \bI_n - 2\nabla \hbmu$ and 
    $\nabla g(\bep) = -2(\nabla\hbmu)\bh$, \eqref{variance-ineq-general} implies 
    \bes
    && \Var(\SURE) 
    \cr &\le &\E\Big[\sigma^2\|\bep - 2\bh + 2(\nabla\hbmu)\bh\|^2
    +\sigma^4\trace\big(\big(\bI_n - 2\nabla \hbmu\big)^2\big)\Big] 
    \cr &=&
    R_{\tSURE}+
    4\sigma^2 \E[(\bep-2\bh)^\top(\nabla \hbmu)\bh + \|(\nabla \hbmu)\bh\|^2]
   \cr & = &R_{\tSURE} + \sigma^2\E\Big[\bep^\top(\nabla\hbmu)\bep
    - (\bep - 2\bh)^\top(\nabla\hbmu)(\bep-2\bh)
    +
    4 \bh^\top\{(\nabla \hbmu)^2-\nabla \hbmu\}\bh
\Big]
   \cr & \le & R_{\tSURE} + \sigma^4n 
    \ees 
    where the first equality follows from \eqref{SURE-for-SURE-calculation},
    the second is simple algebra, and the last inequality follows from 
    ${\bf 0}_{n\times n}\preceq(\nabla \hbmu)^2\preceq \nabla \hbmu\preceq \bI_n$.
    
    (iv) Similar to (iii), $\Rhat'_{\tSURE}/(4\sigma^2) = \|\by - \hbmu\|^2 + \sigma^2 \df  - \sigma^2 n/2$ 
    is a variable of form \eqref{eq:general-random-variable} with 
    $f(\bep) = \bep - \bh$ and $g(\bep) = \bep^\top\bh - \|\bh\|^2-\sigma^2 n/2$. 
    Thus, \eqref{variance-ineq-general} implies 
    \bes
    && \Var(\Rhat'_{\tSURE}/(4\sigma^2)) 
   \cr &\le &\E\Big[\sigma^2\|(\nabla\hbmu - \bI_n)(2\bh-\bep)\|^2
    +\sigma^4\trace\big(\big(\bI_n - \nabla \hbmu\big)^2\big)\Big] 
    \cr &\le &\E\Big[\sigma^2\|2\bh-\bep\|^2
    +\sigma^4\trace\big(\big(\bI_n - \nabla \hbmu\big)^2\big)\Big] 
    \cr & = &\sigma^2\E\Big[4\|\bep-\bh\|^2 + 2\sigma^2(\df - n)
    +\sigma^2\trace\big(\big(\nabla \hbmu\big)^2\big)\Big] 
    \cr & = &(3/4)R'_{\tSURE} + R_{\tSURE}/4 - \sigma^4\E\big[\df\big]
    \cr &=& \E\big[\Rhat''_{\tSURE}\big]
    \ees 
    with the $\Rhat''_{\tSURE}$ in \eqref{SURE4SURE-var-bd}.   
As $\Rhat_{\tSURE} \le \Rhat'_{\tSURE}$, 
$\Var(\Rhat'_{\tSURE}/(4\sigma^2)) \le \E[\Rhat'_{\tSURE}]$.}   
\end{proof}

\begin{proof}[Proof of \Cref{remark:1-lipschitz-positive-symmetric-penalized-estimators}]
    The claim that $\hbmu=\bX\hbbeta(\by)$ is 1-Lipschitz
    is proved in \cite[Proposition 3]{bellec2016bounds}. 
        The symmetry and positivity of
        $\nabla \hbmu$ is proved in \cite[Proposition J.1]{bellec2019second_order_poincare}
        with the argument outlined here:
        $\hbmu(\by)\in\partial u(\by)$ where the function
        $u(\by) = \|\by\|^2/2 - \|\bX\hbbeta(\by) - \by\|^2/2-g(\hbbeta(\by))$
        is convex,
        Alexandrov's theorem on the almost sure second order differentiability
        of convex functions given for instance in 
        \cite[Theorem D.2.1]{niculescu2006convex}
        grants that the Hessian of $u$
        is almost surely symmetric positive semi-definite,
        and the Hessian of $u$ equals $\nabla \hbmu(\by)$ almost surely. 
\end{proof}

\section{Proofs: Confidence regions with SURE}
\label{proof:confidence-regions-SURE}

\begin{proof}[Proof of \Cref{thm:sure-for-sure-tail}
    and \Cref{corollary:exact-quantile-confidence-region}
    ]
    Assume $\sigma=1$ without loss of generality. 
As $f(\by) = \hbmu-\by = (\hbmu-\bmu)-\bep$, 
\bes
\SURE - \|\bmu - \hbmu\|^2
&=& \|\hbmu - \by\|^2 - \|\bmu - \hbmu\|^2 + 2\dv (\hbmu-\by) + n 
\cr &=& \|\bep\|^2 - 2\bep^\top(\hbmu-\bmu)+ 2\dv (\hbmu-\bmu) - n, 
\ees
so that \eqref{th-2-2-1a} is a direct consequence of 
\Cref{thm:second-order-stein}. By Markov' inequality, 
\bes
&& \P\Big\{ |\bep^\top(\hbmu-\bmu) - \dv (\hbmu-\bmu)| \ge v_0\sqrt{n/2}\Big\}
\cr &{\le}&  
\E \big[\|\hbmu -\bmu \|^2 + \trace((\nabla\hbmu(\by))^2)]\big/(v_0^2n/2) \le \eps_n. 
\ees
The conclusion follows from the definition of $v_\alpha$ and the union bound. 

    \Cref{corollary:exact-quantile-confidence-region}
    follows from \Cref{thm:sure-for-sure-tail} with $v_0^2=2\gamma_n^{1/2}$,
    $\eps_n = \gamma_n^{1/2}$ and the continuous mapping theorem.
\end{proof}
}

\begin{proof}[Proof of \Cref{thm:data-driven-gamma_n}]
    {Assume $\sigma=1$ without loss of generality. 
    Set $\df=\trace[\nabla \hbmu]$, $\bh = \hbmu-\bmu$ and 
    $W = \SURE-\E\big[\|\hbmu-\bmu\|^2\big]-\|\bep\|^2+{n}$. 
    As $W$ is of the form \eqref{eq:general-random-variable} with 
    $f(\bep) = -2\bh$ and $g(\bep) =\E[\|\bh\|^2]  - \|\bh\|^2$, 
    \bes
    \Var(W) 
    &\le &\E\big[\| - 2\bh + (\nabla\hbmu)(2\bh)\|^2
    +\trace\big(\big(2\nabla \hbmu\big)^2\big)\big] 
    \cr & \le & \E\big[4\|\bh\|^2+ 4\trace\big(\big(\nabla \hbmu\big)^2\big)\big] 
    \cr & \le & 4\,\E\big[\SURE+ \df\big]. 
    \ees 
    by \eqref{variance-ineq-general}. This gives \eqref{th6-0}. 
    Let $X_n=|W|/\E[W^2]^{1/2}$. We have 
    \bel{bound-W-X_n}
        |W| \le 2X_n\big(\E[{\SURE}+\df]\big)^{1/2}.
    \eel
    As $\SURE+\df$ is of form \eqref{eq:general-random-variable} with 
    $f(\bep) = \bep -3\bh$ and $g(\bep) = - \bep^\top\bh- \|\bh\|^2$, 
    \eqref{variance-ineq-general} yields
    \bes
    \Var(\SURE+\df) 
     &\le &\E\Big[\|\bep - 2\bh + (\nabla\hbmu)(2\bh+\bep)\|^2
    +\trace\big(\big(\bI_n - 3\nabla \hbmu\big)^2\big)\Big].
    \ees
    Using $\trace[(\nabla \hbmu)^2]\le \df$, the second term
        is bounded by $n-6\df + 9\df = n +3\df$.
    Since $(2\bh+\bep)(\nabla \hbmu)^2(2\bh+\bep)
    \le (2\bh+\bep)(\nabla \hbmu)(2\bh+\bep)$, expanding the square yields 
    \bes
       &&\Var(\SURE+\df) 
        \cr &\le &\E\big[\|\bep-2\bh\|^2 + (2\bh+\bep)(\nabla\hbmu)(2\bh+\bep) 
    + 2(\bep - 2\bh)^\top(\nabla\hbmu)(2\bh+\bep)\big]
    \cr && +\E\big[n+3\df\big].  
    \cr & = & \E\big[n+3\df + \|\bep-2\bh\|^2 + 4\bep^\top(\nabla \hbmu)\bep\big]
       -\E\big[(\bep-2\bh)^\top(\nabla \hbmu)(\bep-2\bh)\big].  
   \cr & \le & \E\big[4\SURE + (6n-\df)\big]
    \ees 
    where the last inequality follows from $\bep^\top(\nabla\hbmu)\bep\le\|\bep\|^2$
    and Stein's formulae $\E[-4\bep^\top\bh]=\E[-4\df]$ and $\E[\|\bh\|^2]=\E[\SURE]$. 
    Hence} there exists a random variable $Y_n\ge0$ with $\E[Y_n^2]\le 1$
    such that almost surely
    $\E[\SURE+\df] \le \SURE + \df + {2Y_n \E[\SURE+\df]^{1/2}+Y_n\sqrt{6n}.}$
    By completing the square,
    \bel{bound-Y_n}
    &&
    (\E[\SURE+\df]^{1/2} -Y_n)^2
    \le \SURE + \df + {Y_n^2 + Y_n\sqrt{6n}.}
    \eel
    Combining \eqref{bound-W-X_n} and \eqref{bound-Y_n} above
    we get almost surely
    \bes
    |W|
      &\le&
    2X_n\big[
    (\SURE+\df)_+^{1/2}
    + {2Y_n + \sqrt{Y_n}(6n)^{1/4}}
    \big]
    \ees
    which  is equivalent to \eqref{th6-1}.
    As $X_n\le 1/\sqrt{\beta_1}$ and $Y_n\le 1/\sqrt{\beta_2}$ with
    probability at least probability $1-\beta_1-\beta_2$, 
    the conclusions follow. 

    For the estimation of $\|\hbmu-\bmu\|^2$, we set  
    $W = \SURE - \|\hbmu-\bmu\|^2 - \|\bep\|^2+n$ with 
    $\E[W^2]\le 4\E[\SURE + \df]$
    in virtue of $\trace(\{\nabla \hbmu\}^2)\le \df$ by 1-Lipschitzness
    of $\hbmu$ so that the upper bounds for $\E\big[\SURE+ \df\big]$ still apply.
\end{proof}
    
\section{Proofs of \Cref{THM:ORACLE-INEQ}, \Cref{prop:s-star-cannot-removed-for-sure-TUNED} and \eqref{eq:Q-aggregation-oracle-inequality}}
\label{sec:proof:oracle-ineq-SURE}

The proof of \Cref{THM:ORACLE-INEQ} requires the following lemma. 

\begin{lemma}
    \label{lemma:oracle-inequality}
    Consider the sequence model $\by = \bmu + \bep$
    with $\bep\sim N({\bf 0},{\sigma^2}\bI_n)$.
    Let $\hbmu^{(1)},\hbmu^{(2)}$ be two estimators
    that are {$L$-Lipschitz} functions of $\by$
    and let $f(\by) = \hbmu^{(2)} - \hbmu^{(1)}$.
    Let $\tbmu$ be the estimator among $\{\hbmu^{(1)},\hbmu^{(2)} \}$
    with the smallest $\SURE$. Then either
    $$
    \E[ \Delta^4]
    \le 8{\sigma^4} \E \trace[ \{ \nabla f(\by)\}^2 ]
    \qquad
    \text{ or }
    \qquad
    \E[ \Delta^2] \le {8(\sqrt{2}L+1)\sigma^2}
    $$
    holds, where
    $\Delta =  \|\tbmu - \bmu\| - \min_{j=1,2}\|\hbmu^{(j)}-\bmu\|$.
\end{lemma}

\begin{proof}[Proof of \Cref{lemma:oracle-inequality}]
Let $\xi = \|\hbmu^{(1)}-\bmu\|^2 - \|\hbmu^{(2)}-\bmu\|^2
- \SURE{}^{(1)} + \SURE{}^{(2)}
= 2\dv f(\by) - 2\bep^\top f(\by)$. 
By definitions of $\tbmu$ and $\Delta$,
    \bes
    |\xi| &\ge&\|\tbmu - \bmu\|^2 - \min_{j=1,2}\|\hbmu^{(j)} - \bmu\|^2
    \cr &=& \big\{\|\hbmu^{(1)} - \bmu\| + \|\hbmu^{(2)} - \bmu\|\big\}\Delta
    \cr &\ge& (\|f(\by)\|\Delta)\vee\Delta^2.
    \ees

If $\sigma^2\E \trace[ \{ \nabla f(\by)\}^2 ] \ge \E[\|f(\by)\|^2]$
then $\E[\Delta^4]\le\E[\xi^2]$ and the identity \eqref{th-1-sigma} for $\E[\xi^2]/4$
yields the bound on $\E[\Delta^4]$. 

{If now $\sigma^2\E \trace[ \{ \nabla f(\by)\}^2 ] \le \E[\|f(\by)\|^2]$,
as $\by\to\|f(\by)\|$ is $2L$-Lipschitz
\bes
\E[\|f(\by)\|^2]\E\big[\Delta^2\big]
&\le& \E\big[\|f(\by)\|^2\Delta^2\big] + \big\{\Var(\|f(\by)\|^2)\E\big[\Delta^4\big]\big\}^{1/2}
\cr &\le& \E\big[\xi^2\big] + \big\{\Var(\|f(\by)\|^2)\E\big[\xi^2\big]\big\}^{1/2}
\cr &\le& 8\sigma^2\E[\|f(\by)\|^2]+\sqrt{(4L)^2 8}\sigma^2\E[\|f(\by)\|^2],
\ees
by the Cauchy-Schwarz inequality for the first inequality and the
Gaussian Poincar\'e inequality $\Var[\|f(\by)\|^2]\le
\E[\|2\{\nabla f(\by)\}f(\by)\|^2]
    \le
(4L)^2\E[\|f(\by)\|^2]
$
for the second. Hence $\E[\Delta^2] \le 8(\sqrt{2}L+1)\sigma^2$.}
\end{proof}

\begin{proof}[Proof of \Cref{THM:ORACLE-INEQ}]
    Let $j_0= \argmin_{j=1,...,m} \E \|\hbmu^{(j)} - \bmu\|$.
    For $k\in [m]$, let  
    $$\Delta_k = I_k\big(\|\hbmu^{(k)}   -\bmu \| - \|\hbmu^{(j_0)}-\bmu\|\big)_+,$$
    where $I_k$ is the indicator of the event that the $\SURE$
    of $\hbmu^{(k)}$ is smaller than the $\SURE$ of $\hbmu^{(j_0)}$.
    Then by \Cref{lemma:oracle-inequality} 
    we have $\E[A_k] \le 1$
    with
    $$A_k = \min\{\Delta_k^4/(8\sigma^4s^*),\; \Delta_k^2/(8\sigma^2(\sqrt 2 L +1))\}$$ and $X=m^{-1}\max_{k\in[m]}A_k$ has also $\E[X]\le 1$.
    Then almost surely,
    \begin{equation}
        \label{eq:almost-sure-X-delta-k-hat}
        \|\tbmu-\bmu\|-\|\hbmu^{(j_0)}-\bmu\|
        \le \Delta_{\hat k}
        \le \sigma \max\{(8s^*mX)^{1/4},(8(\sqrt{2} L +1)mX)^{1/2}\} 
    \end{equation}
    so that $\P(X>1/\alpha)\le \alpha$ yields (i). 

    (ii) Set 
    $W_k  =
      (
      \|\hbmu^{(j_0)} - \bmu\|
      -\|\hbmu^{(k)} - \bmu\|
      )_+$.
    For each $k\in [m]$,
    the function $\by \to \|\hbmu^{(j_0)} - \bmu\| - \|\hbmu^{(k)} - \bmu\|$
    is $2L$-Lipschitz with negative expectation so that
    $\P(W_k > 2L {\sigma}\sqrt{2x} ) \le e^{-x}$ for all $x>0$ by Gaussian concentration
    and
    $\P(\max_{k\in[m]} W_k > 2L{\sigma} \sqrt{2\log(m/\delta)}) \le \delta$
    by the union bound.

    (iii) is obtained using \eqref{eq:almost-sure-X-delta-k-hat}, $\E[X]\le 1$
    and $\E[\max_{k\in[m]}W_k^2]\le 8L^2\sigma^2\log(em)$ by integration.

    (iv)
    For \eqref{eq:oracle-ineq-SURE-squared-risk},
    let $k_0=\argmin_{j\in[m]}\E[\|\hbmu^{(j)}-\bmu\|^2]$.
    For all $j,k\in[m]$ we have
    for $\xi_{j,k} = \|\hbmu^{(k)}-\bmu\|^2 - \SURE^{(k)} - \|\hbmu^{(j)}-\bmu\|^2 + \SURE^{(j)}$ that
    \bes
        \E[ \xi_{j,k}^2]
        &=& 4\E[\sigma^2\|\hbmu^{(j)}-\hbmu^{(k)}\|^2 + {4}\sigma^4\trace(\{\nabla \hbmu^{(k)}-\nabla \hbmu^{(j)}\}^2)]
    \\ &\le& 16L^2\sigma^4 n + 16L^2\sigma^4 n = 32L^2\sigma^4 n
    \ees
    by assumption.
    Since $G = \max_{k\in[m]}\xi_{k_0,k}^2/(32L^2 \sigma^4 m n)$ has $\E[G]\le 1$ 
    and $\|\tbmu-\bmu\|^2 - \|\hbmu^{(k_0)}-\bmu\|^2 \le (32 G m n)^{1/2} L \sigma^2$
    holds a.s., we get \eqref{eq:oracle-ineq-SURE-squared-risk}.
\end{proof}

\begin{proof}[Proof of \Cref{prop:s-star-cannot-removed-for-sure-TUNED}]
    Let $\bv\in\R^n$ with $\|\bv\|^2 = \sigma^2 \sqrt n$.
    Choose 
    $$\bmu={\bf 0},
    \qquad
    \hbmu^{(1)}(\by)={\bf 0},
    \qquad \hbmu^{(2)}(\by)=\bv + G(\by)
    $$
    where $G:\R^n\to\R^n$ with $G(\by)_i = g(y_i)$ and
    $g$ is defined as the only function that is
    $2^{n+1}\sigma$ periodic,
    symmetric ($g(-u)=g(u)$) and
    with $g(u)=\sigma(\frac u \sigma \wedge (2^{n}-\frac u \sigma))$ on $[0,2^{n}\sigma]$.
    Observe that $g$ and $G$ are both 1-Lipschitz.
    Furthermore, $\P(g'(y_i)=\pm 1)=1/2$ by symmetry of $g$
    and $\dv G(\by)=\sum_{i=1}^n r_i$
    where $(r_1,...,r_n)$ are iid with $\P(r_i=\pm 1)=1/2$.
    Since $\hbmu^{(1)}$ is the oracle, i.e., it has smaller risk than $\hbmu^{(2)}$,
    we now study the event $\Omega=\{\SURE^{(2)}<\SURE^{(1)} \}$ in which
    $\tbmu$ selects the worse estimator. Event $\Omega$
    can be rewritten
    $\{\|\bv+G(\by)\|^2 - 2\bep^\top(\bv+G(\by)) + 2\sigma^2 \dv G(\by) < 0\}$.
    We have $\mathbb P\{\dv G(\by) \le - \sqrt n\}\ge C_0>0$ by the reverse of
    Hoeffding inequality given in \cite[Theorem 7.3.2]{matousek2008probabilistic}, while
    $\|\bv+G(\by)\|^2 - 2\bep^\top(\bv+G(\by))=\sigma^2\sqrt n(1+o_\P(1))$ since $\|\bv\|^2=\sigma^2\sqrt n$,  $|G|\le \sigma 2^{-n}$ and $\bep^\top\bv = O_\P(\sigma\|\bv\|)$.
    This proves that $\P(\Omega)\ge C_3>0$ for instance with $C_3 = C_0/2$ and any $n\ge C_1$
    for large enough $C_1$. On $\Omega$, we have $\tbmu=\hbmu^{(2)}$ as well as
    $$
    \|\tbmu-\bmu\| - \|\hbmu^{(1)}-\bmu\|
    = \|\hbmu^{(2)}-\bmu\|
    = \|\bv + G(\by)\|
    \ge C_2 \sigma n^{1/4}
    .$$ 
\end{proof}

\begin{proof}[Proof of \eqref{eq:Q-aggregation-oracle-inequality}
    assuming $\by\to\hbmu^{(j)}(\by)$ is $L$-Lipschitz for all $j$
    ]
Let $k\in[m]$ be fixed.
Proposition 3.2 in \cite{bellec2014affine} states that $\tbmu_Q$
satisfies
\bes
&&\|\tbmu_Q-\bmu\|^2
-
\|\hbmu^{(j_0)}-\bmu\|^2
\\
&\le&
\max_{k\in[m]}
2\{ \bep^\top (\hbmu^{(k)} - \hbmu^{(j_0)})  -  \sigma^2\dv(\hbmu^{(k)}-\hbmu^{(j_0)}) - \|\hbmu^{(j_0}-\hbmu^{(k)}\|^2/4\}.
\ees
Let $W_{j_0,k}$ be the random variable inside the maximum, which is of the form
\eqref{eq:general-random-variable}
with $f(\bep)=(\hbmu^{(k)}-\hbmu^{(j_0})$ and $g(\bep)=\|f(\bep)\|^2/4$.
Then $G=\max_{j\in[m]}(W_{j_0,k}-\E[W_{j_0,k}])_+^2/(m\Var[W_{j_0,k}])$ has $\E[G]\le 1$
and 
$\Var[W_{j_0,k}]
\le
\sigma^2 \|(\bI_n - \frac 1 2 \nabla f)f(\bep)\|^2 + \sigma^4 s^*
\le (1+L)^2\sigma^2 \E[\|\hbmu^{j_0}-\hbmu^{(k)}\|^2] + \sigma^4s^*
$ by \eqref{variance-ineq-general}.
Then almost surely
\bes
    W_{j_0,k}
    & \le & \sqrt{Gm}\Var[W_{j,k}]^{1/2} - \E[\|\hbmu^{j}-\hbmu^{(k)}\|^2/4]
    \cr
    &\le & \sigma^2 \sqrt{G m s^*} + (1+L)\sigma \sqrt{G} \E[\|\hbmu^{j}-\hbmu^{(k)}\|^2]^{1/2} - \E[\|\hbmu^{j}-\hbmu^{(k)}\|^2/4]
    \cr &\le&
    \sigma^2\sqrt{Gm s*} + (1+L)^2\sigma^2 G.
\ees
The proof is complete using $\E[G]\le 1$
and $\E[\|\hbmu^{(j_0)}-\bmu\|^2]-\min_{j\in[m]}\E[\|\hbmu^{(j)}-\bmu\|^2]
\le \sigma^2L^2$ by definition of $j_0$ and $L$-Lipschitzness of
$\hbmu^{(j_0)}$.
\end{proof}

\section{Strictness of the KKT conditions}
\label{sec:proof-prop-4-1}

\begin{proof}[Proof of \Cref{prop-4-1}]
    Assume that $\bX_B$ has rank strictly less than $|B|$. Then there must exist
    some $j\in B$ and $A\subseteq B\setminus\{j\}$ with $\bx_j = \sum_{k\in A} \gamma_k \bx_k$
    and $\rank(\bX_A)=\min(|A|,n)$.
    By the definition of $B$
    $${\lambda n \delta_j}  =  \bx_j^\top(\by - \bX{\lasso})
    = \lambda n
    \sum_{k\in A}
    \gamma_k
    \delta_k
    $$
    where $\delta_k = \bx_k^\top(\by - \bX{\lasso})/(\lambda n) \in\{-1,1\}$.
    This is impossible by \Cref{assum:bX} on $\bX$. Hence $\bX_B$ has rank $|B|$.
    For the uniqueness, consider two Lasso solutions ${\lasso}$ 
    and ${\hat \bb}$ of \eqref{lasso}.
    It is easily seen that $\bX{\lasso}=\bX{\hat \bb}$ 
    by the strict convexity of the squared loss in $\bX\bb$ in \eqref{lasso};
    actually the function $\by\to \bX{\lasso}$ is 1-Lipschitz (cf. for instance \cite{bellec2016slope}).
    Furthermore both ${\lasso},{\hat \bb}$
    must be supported on $B$. Hence $\bX_B({\lasso})_B=\bX_B{\hat \bb}_B$
    which implies that ${\hat \bb}_B=({\lasso})_B$ because $\bX_B$ has rank $|B|$.

    It remains to show that for any $j\notin \Shat$,
    the KKT conditions on coordinate $j$ holds strictly with probability one. 
    {As $\bX_B$ has rank $|B|$, it suffices to consider the case of $|\Shat|<n$. 
    By the KKT conditions, $({\lasso})_{\Shat} = (\bX_{\Shat}^\top\bX_{\Shat})^{-1}
    \{\bX_{\Shat}^\top\by - n\lam\sgn(({\lasso})_{\Shat})\}$. As 
    $\P\big[\bv^\top\by = c\big]=0$ for all deterministic $\bv\neq{\bf 0}$ and real $c$, 
    \bes
    \E\Big(\P\Big[\bx_j^\top\big\{\by - \bX_S(\bX_{S}^\top\bX_S)^{-1}
    (\bX_S^\top\by - n\lam\bu_S)/n\big\} = \pm \lam\Big|\bX \Big]\Big) = 0 
    \ees
    for all deterministic $\{S,{j},\bu\}$ satisfying 
    $\rank(\bX_S)=|S|<n$, $\rank(\bX_{S\cup\{{j}\}})=|S|+1$ and 
    $\bu_S\in\{\pm 1\}^S$. Hence, $\P[|B|>|\Shat|]=0$, which means that} the KKT conditions
    of ${\lasso}$ must hold strictly with probability one.  
\end{proof}

\section{Proof: bound on the variance of $|\hat S|$}
\label{proof:bound-sparsity-general}

\begin{proof}[Proof of \Cref{thm:variance-size-model-lasso}]
    Assume $\sigma=1$ without loss of generality due to scale invariance. 
    The first claim follows from (\ref{diff-Lasso}) and the discussion leading to it, 
    combined with \Cref{prop:variance-divergence}.
    Next we first use the rough bounds $|\Shat|\le n$
    and $\|\bP_{\Shat} \bep\| \le \|\bep\|$ to obtain $\Var[|\Shat|] \le 2n$.
    For the right term of the minimum, for a fixed $A\subset[p]$,
    the random variable $\|\bP_A\bep\|^2$ has chi-squared distribution with at most
    $|A|$ degrees of freedom and a classical tail bound (cf. for instance \cite[Lemma 1]{laurent2000adaptive})
    states that
    $$
    \P(\|\bP_A\bep\|^2> 2|A| + 3x)
    \le
    \P(\|\bP_A\bep\|^2>|A| + 2\sqrt{x|A|} + 2x)\le e^{-x}.$$
    Consequently, by the union bound over all ${p \choose m}\le (\frac{ep}{m})^m$ supports $A$ of size $m$,
    $$\P\left(
        \max_{A\subset[p]:|A|=m }\|\bP_A\bep\|^2>2m + 3\left(m \log\left(\frac{ep}{m}\right) + x\right)
    \right)\le e^{-x}.$$
    By a second union bound over all possible support sizes $m=1,...,p$,
    \begin{equation*}
    \P\left(
        \max_{A\subset[p]}
        \left\{
                \|\bP_{A}\bep\|^2-2|A| - 3\left(|A| \log\left(\frac{ep}{|A|\vee 1}\right)\right)
        \right\}
        > 3(\log p  + x) 
    \right)\le e^{-x}.
    \end{equation*}
    Finally,
    let 
    $X = (1/3)
         \max_{A\subset[p]}
         \{
                \|\bP_{A}\bep\|^2 - 2|A| - 3
                (
                |A| \log(\frac{ep}{|A|})+ \log p
                )
        \}
        $
    so that
    $\P(X>x)\le e^{-x}$ holds.
    The identity $\E[\max(X,0)]=\int_0^{\infty}\P(X>x) dx\le 1$ yields
    \bel{inequality-bP-bep-any-Shat}
    \E[
        \|\bP_{\Shat}\bep\|^2
    ]
    &\le&
    2 \E|\Shat|
    +
    3\E\left[|\Shat| \log\left({ep} {\big/} \{|\Shat|\vee 1\}\right)+ \log(e p)\right] 
    \cr &\le& 
    2 \E|\Shat|
    +
    4\E\left[|{\Shat}| \log\left({ep} {\big/} \{|\Shat| {\vee 1}\}\right)\right].
    \eel
    The proof is complete 
    as the second inequality in (\ref{th-model-size-2}) follows 
    from the concavity of the function $x \to x\log(ep/{(x\vee 1)})$.
\end{proof}

\section{Preliminaries for bounds on $\E|\Shat|$}

\begin{lemma}
    \label{lemma:calculation-RE}
    {Let} $Z$ be a standard normal random variable. Then, 
    \bes
    \P\Big[ Z > t\Big] \le \frac{e^{-t^2/2}}{(2\pi t^2+4)^{1/2}},\quad \forall\ t\ge 0, 
    \ees
    \bes
    {\E\Big[(|Z|- t)_+\Big] \le \frac{2e^{-t^2/2}}{(2\pi)^{1/2}(t^2+1)},\quad \forall\ t\ge 0,}
    \ees
    and
    \bel{bound-second-moment-sharp-at-0-and-infty}
    \E\Big[(|Z|-t)_+^2\Big] \le \frac{4e^{-t^2/2}}{(t^2+2)(2\pi t^2+4)^{1/2}},\quad \forall\ t\ge 0.
    \eel
\end{lemma}

\paragraph{Remark} Compared with the usual tail probability bounds for standard Gaussian, 
the upper bounds in \Cref{lemma:calculation-RE} is sharp at both $t=0$ and $t\to\infty$.

\begin{proof}
{Let} $t>0$. Let $\varphi(t)$ and $\Phi(t)$ respectively 
be the density and cumulative distribution function of $Z$.
With $u = tx+x^2/2$ and $du=(t+x)dx$, 
\bes
\frac{\Phi(-t)}{\varphi(t)}=\int_0^\infty e^{-tx-x^2/2}dx 
= \int_0^\infty \frac{e^{-u}du}{(t^2+2u)^{1/2}}
= \int_0^\infty f_t (u^{-1/2})e^{-u}du, 
\ees
where $f_t(x) = (t^2+2/x^2)^{-1/2} = x(t^2x^2+2)^{-1/2}$ is a concave function of $x$. Thus, 
\bes
\Phi(-t)/\varphi(t)
\le f_t\big(\Gamma(1/2)\big) 
= f_t\big(\sqrt{\pi}\big) = (t^2+2/\pi)^{-1/2}. 
\ees
This gives the tail probability bound. 
From the well known $t(1+t^2)^{-1}\varphi(t)\le \Phi(-t)$, we also have 
\bes
\E\big[(|Z|-t)_+\big] = 2\big\{\varphi(t)-t\Phi(-t)\big\} \le 2\varphi(t)/(1+t^2). 
\ees
Define $J_k(t)=\int_0^\infty x^k e^{-x-x^2/(2t^2)}dx$. 
For the second tail moment, we have 
$\E\big[(|Z|-t)_+^2\big] = 2\varphi(t)J_2(t)/t^3 = 2\Phi(-t)J_2(t)/\{t^2J_0(t)\}$. 
As in Proposition 10 (i) in \cite{sun2013sparse} and its proof, 
\eqref{bound-second-moment-sharp-at-0-and-infty} follows from
$J_2(t)/J_0(t)\le 1/(1/2+1/t^2)$ due to the recursion 
$J_{k+1}(t)+t^{-2}J_{k+2}(t) = (k+1)J_k(t)$ {for $k\ge 0$.}
\end{proof}

\section{Upper bound on the sparsity of
the Lasso under the SRC}
\label{sec:proof-upperbound-expected-sparsity-SRC}

\begin{proof}[Proof of \Cref{proposition:SRC-upperbound-expected-sparsity}]
The SRC \eqref{SRC} can be written as 
\begin{equation}
    \label{SRC-rewritten}
    \Big(\max_{B\subset[p]:|{B}\setminus S|\le m}\phi_{\rm cond}\big(\bX_{B}^\top\bX_{B}\big)-1\Big)
    \big(|S| + \epsilon_2k\big)+\epsilon_1k\le 2(1-\eta)^2 m
    .
\end{equation}
    Let $\bg=(n\lam)^{-1}\bX^\top\bP_S^\perp \bep$ with elements $g_j\sim N(0,\sigma_j^2)$ satisfying  
$1/\sigma_j\ge t_\tau = \eta^{-1}\sqrt{(1+\tau)2\log(p/k)}$.
Let $C_0=\max_{B:|B\setminus S|\le m}\phi_{\rm cond}\big(\bX_{B}^\top\bX_{B}\big)-1$ and 
    $$
    \Omega=\big\{4(1-\eta)\|(|\bg|-\eta)_+\|_1+C_0\|(|\bg|-1)_+\|_2^2 ~ < ~ \eps_1k + C_0\eps_2k\big\}.
    $$ 
    It follows from \Cref{lemma:modified-SRC-lemma} below that 
    $|\Shat\setminus S| < m$ in this event $\Omega$.  Applying \Cref{lemma:calculation-RE}, we find that 
    \bes
    4(1-\eta)\E\big[\|(|\bg|-\eta)_+\|_1\big] \le \frac{8(1-\eta)pe^{-(\eta t_\tau)^2/2}}{\sqrt{2\pi}t_\tau(1+(\eta t_\tau)^2)} 
    \le \frac{\eps_1 k(k/p)^\tau}{1+(\eta t_0)^2} 
    \ees
    with $\eps_1 = 8(1-\eta)/\{\sqrt{2\pi}t_0\}$ and 
    \bes
    \E\big[\|(|\bg|-1)_+\|_2^2\big] \le \frac{4e^{-t_\tau^2/2}}{t_\tau^2(t_\tau^2+2)(2\pi t_\tau^2+4)^{1/2}} \le 
    \frac{\eps_2 k(k/p)^\tau}{1+(\eta t_0)^2} 
    \ees
    with $\eps_2 = 4/\{t_0^2(2\pi t_0^2+4)^{1/2}\}$. Thus, $\P\{\Omega\}$ is no smaller than 
    $1 - (k/p)^\tau/(1+(\eta t_0)^2)$. 
    Finally $\E|\Shat|$ is bounded from above by
    $|S|+m+\E[I_{\Omega^c}|\Shat|]$ so that combining
    $\E[I_{\Omega^c}|\Shat|]\le p \P(\Omega^c)$ with the previous bound on $\P(\Omega^c)$
    gives the upper bound on $\E|\Shat|$.
\end{proof}

{The following lemma is a slight modification of} \cite[Proposition 7.4]{bellec_zhang_future_lasso-dof-adjustment}.

\begin{lemma}[Deterministic lemma]
    \label{lemma:modified-SRC-lemma}
    Let $\lambda,\epsilon_1,\epsilon_2>0$ and $\eta\in(0,1)$.
    Let $\hbbeta$ be the Lasso \eqref{lasso} with $\by=\bX\bbeta+\bep$
    and $S=\supp(\bbeta)$. If for
    $\bg = \bX^\top\bP_S^\perp\bep/(n\lam)$ and
    $C_0=\max_{B:|B\setminus S|\le m}\phi_{\rm cond}\big(\bX_{B}^\top\bX_{B}\big)-1$
    we have
$$ 
4(1-\eta)\big\|(|\bg|-\eta)_+\big\|_1 
+C_0\big\|(|\bg|-1)_+\big\|_2^2 ~<~ \epsilon_1k + C_0\epsilon_2k 
$$
and the SRC \eqref{SRC} holds, then $|\Shat\setminus S| < m$.
\end{lemma}

\begin{proof}
    The SRC \eqref{SRC} can be written as  \eqref{SRC-rewritten}.
Let $\bbetabar$ the oracle LSE satisfying $\supp(\bbetabar)\subseteq S$ 
and $\bX\bbetabar = \bP_S\by$. Let ${B}$ satisfy
\bes
S\cup\Shat \subseteq {B}\subseteq 
S\cup \big\{j: |\bx_j^T(\by-\bX\hbbeta)/n|=\lam\big\}. 
\ees
Let $\bSigmabar = \bX^\top\bX/n$, $B_1 = {B} \setminus S$, 
$\bu = (\hbbeta - \bbetabar)/\lam$ and $\bv = \bSigmabar \bu$. 
By algebra, 
\bes
\bv_{B}^\top\bSigmabar_{{B},{B}}^{-1}\bv_{B}
+ \bv_{B_1}^\top\big(\bSigmabar_{{B},{B}}^{-1}\big)_{{B_1},{B_1}}\bv_{B_1}
= \bv_S^\top\big(\bSigmabar_{{B},{B}}^{-1}\big)_{S,S}\bv_S+ 2\bu_{B_1}^\top\bv_{B_1}. 
\ees
Let $\bs = \bX^\top(\by - \bX\hbbeta)/(n\lam) = \bg-\bv$. 
By the KKT conditions, $s_j=\sgn(u_j)$ for $j\in B_1$, so that 
$\bu_{B_1}^\top\bv_{B_1}
= \sum_{j\in {B}\setminus S} |u_j|(s_jg_j-1) \le \bv_{B}^\top\bSigmabar_{{B},{B}}^{-1}\bw_{B}$, 
where $\bw$ is the vector with elements $w_j = I\{j\in {B_1}\}s_j(s_jg_j-1)_+$. By algebra, 
\bes
&& (\bv-\bw)_{B}^\top\bSigmabar_{{B},{B}}^{-1}(\bv-\bw)_{B}
+ \bv_{B_1}^\top\big(\bSigmabar_{{B},{B}}^{-1}\big)_{{B_1},{B_1}}\bv_{B_1}
\cr &\le& \bv_S^\top\big(\bSigmabar_{{B},{B}}^{-1}\big)_{S,S}\bv_S+ \bw_{B}\bSigmabar_{{B},{B}}^{-1}\bw_{B}. 
\ees
It follows that 
\bes
\|\bv_S\|_2^2+\|\bv_{B_1}-\bw_{B_1}\|_2^2 + \|\bv_{B_1}\|_2^2 \le \phi_{\rm cond}\big(\bSigmabar_{{B},{B}}\big)
\big(\|\bv_S\|_2^2+\|\bw_{B_1}\|_2^2\big).
\ees 
Moreover, as $v_j-w_j = -s_j(s_jg_j-1)_-$ for $j\in {B_1}$, $\|\bv_{B_1}\|_2^2=\|\bv_{B_1}-\bw_{B_1}\|_2^2+\|\bw_{B_1}\|_2^2$. 
This and the inequality above imply 
\bes
2\|\bv_{B_1}-\bw_{B_1}\|_2^2 \le \phi_{\rm cond}\big(\bSigmabar_{{B},{B}}-1\big)
\big(\|\bv_S\|_2^2+\|\bw_{B_1}\|_2^2\big) 
\ees
As $\bv = \bg-\bs$ and $\bg_{S}=0$, $\|\bv_S\|_2^2=|S|$.
We also have $\|\bw_{B_1}\|_2^2 \le \sum_{j=1}^p (s_jg_j-1)_+^2$ 
so that $2\|\bv_{B_1}-\bw_{B_1}\|_2^2\le C_0(|S| + \|(|\bg| -1)_+\|_2^2)$. 
As $s_j=\sgn(u_j)$ and $|v_j-w_j| = (1-s_jg_j)_+$ in ${B_1}$, 
\bes
2(1-\eta)^2|{B_1}| &\le& 2\|\bv_{B_1}-\bw_{B_1}\|_2^2+4(1-\eta)\sum_{j=1}^p(s_jg_j-\eta)_+
\cr &\le& 
C_0 \big(|S| + \|(|\bg|-1)_+\|_2^2 \big) + 4(1-\eta) \|(|\bg|-\eta)_+\|_1
\cr &<& 
C_0 \big(|S| + \epsilon_2k\big) +\epsilon_1k. 
\ees
Thus, by \eqref{SRC-rewritten}, $|{B_1}|\le m$ implies $|{B_1}| < m$. 
As $|{B_1}|$ is {allowed to} change one-at-a-time 
along the Lasso path from $\lam=\infty$ and the {condition on $\bg$} is monotone in $\lam$, 
$|{B_1}|  <  m$ holds for all penalty levels satisfying the condition on {$\bg$}. 
For details of this argument see \cite[Proof of Lemma 1]{zhang10-mc+}. 
\end{proof}

\section{Upper bound on the sparsity of
the Lasso under the RE condition}
\label{sec:proof-upperbound-expected-sparsity-RE}

\begin{proof}[Proof of \Cref{thm:upperbound-expected-sparsity-RE}]
Let $\mu=(1+\tau)\sigma\sqrt{2\log(ep/(s_0\vee 1))/n}$ and $\lambda=(1+\gamma)\mu$.
For each $j\in[p]$ set $g_j = (\tau + 1)\bep^\top \bX\be_j/n$. 
By the KKT conditions of the Lasso we have
$$ (\tau \bep^\top\bX\bh + \|\bX\bh\|^2)/ n
\le \bg^\top\bh - \lambda \|\bh_{S^c}\|_1 - \lam\,\sgn(\bbeta_S)^\top\bh_S, 
$$
where $\bh={\lasso}-\bbeta$.
Define $T= \{\bu\in\R^p:\|\bu_{S^c}\|_1 < c_0\sqrt{s_0\vee 1}\|\bu\|\}$ as well as
\begin{align*}
& f_{\lambda,\mu}(\bep) 
=\sup_{\bu\in T}
\frac{\big(\bg^\top\bu - \mu \|\bu_{S^c}\|_1 - \lam\,\sgn(\bbeta_S)^\top\bu_S\big)_+}{\|\bu\|},\\
& g_{\lambda,\mu}(\bep) = 
 \sup_{\bu\not\in T,\|\bX\bu\|>0}\frac{\big(\bg^\top\bu - \mu \|\bu_{S^c}\|_1 - \lam\,\sgn(\bbeta_S)^\top\bu_S 
- \gamma \lam c_0\sqrt{s_0\vee 1}\|\bu\|\big)_+}{\|\bX\bu\|/\sqrt{n}}
\end{align*}
where $a_+=\max(a,0)$ for any real $a$.
Let $\Omega$ be the event $\Omega=\{ \bh \in T\}$ and $I_\Omega$ be its indicator function.
Using the elementary inequality $2ab - b^2 \le a^2$ we get
$$ 
(2\tau \bep^\top\bX\bh + \|\bX\bh\|^2)/n
\le 
I_{\Omega} f_{\lambda,\lam}(\bep)^2/\RE(S,c_0)^2
    +
    I_{\Omega^c} g_{\lambda,\lambda}(\bep)^2.
$$
We now bound the expectation $\E[f_{\lambda,\mu}(\bep)^2]$.
By simple algebra on each coordinate and the Cauchy-Schwarz inequality,
\begin{align*}
f_{\lambda,\mu}(\bep)^2 = 
\sum_{j\in S} (g_j - \lambda \sgn(\beta_j))^2
+\sum_{j\notin S} (|g_j| - \mu)^2.
\end{align*}
Each $g_j$ is centered, normal with variance at most $\omega^2=(1+\tau)^2\sigma^2/n$,
hence \Cref{lemma:calculation-RE} implies that 
$\E[f_{\lambda,\mu}(\bep)^2] \le
s_0(\lambda^2 + \omega^2)+(s_0\vee 1)\omega^2$, 
which is then bounded by $2^{-1}(\gamma c_0)^2(s_0\vee 1)\lambda^2$ 
by the condition on $c_0$.

Note that by construction, the function $\bep\to g_{\lambda,\mu}(\bep)$
is $((1+\tau)/\sqrt n)$-Lipschitz, so that by the Gaussian concentration inequality (see, e.g., \cite[Theorem 10.17]{boucheron2013concentration}),
$$
\E[g_{\lambda,\mu}(\bep)^2]
=\int_0^{+\infty} \P\left[g_{\lambda,\mu}(\bep)> \sqrt t\right] dt
\le
\int_0^{+\infty} \P\left[\omega N(0,1)  > \sqrt t\right] dt = \omega^2/2, 
$$
provided that the median of $g_{\lambda,\mu}(\bep)$ is zero.
We now prove that the median is indeed zero.
The event $\{f_{\lambda,\mu}(\bep)^2\le 2 \E[f_{\lambda,\mu}(\bep)^2]\}$ has probability
at least $1/2$ thanks to Markov's inequality.
Furthermore, we proved above that 
$2 \E[f_{\lambda,\mu}(\bep)^2]\le (c_0\gamma)^2(s_0\vee 1)\lambda^2$.
On this event of probability at least $1/2$,
for any $\bu\in T$ we have
\begin{align*}
\big(\bg^\top\bu - \mu \|\bu_{S^c}\|_1 - \lam\,\sgn(\bbeta_S)^\top\bu_S 
- (c_0\gamma)\lam\sqrt{s_0\vee 1}\|\bu\|\big)_+=0.
\end{align*}
We have established that the median of $g_{\lambda,\mu}$ is non-positive
and the proof is complete.
\end{proof}

\section{Consistency of SURE for SURE in the Lasso case}
\label{sec:proof-consistency-SURE-lasso}

\begin{proof}[Proof of \Cref{prop:consistency-SURE-for-SURE-Lasso}]
%
%
    {(i)-(iv) follow directly from \Cref{th-consistency}. For (v),} 
    Markov's  inequality 
    with the bound in (iii)
    followed by the bound $\sigma^4n\le\E[\hat R_{\tSURE}]$ from (ii)
    implies
    $$|\hat R_{\tSURE}-\E[\hat R_{\tSURE}]|^2
    \le 4(n\alpha)^{-1} \sigma^4 n \E[\hat R_{\tSURE}]
    \le 4(n\alpha)^{-1} \E[\hat R_{\tSURE}]^2
    $$
    with probability at least $1-\alpha$,
    and $\E[\hat R_{\tSURE}]\le (1-4(n\alpha)^{-1/2})^{-1}\hat R_{\tSURE}$ on this event
    by the triangle inequality.
    The proof is completed using
    $\P(|\SURE-\|\bX\lasso-\mu\|^2|\le \gamma^{-1} \E[\hat R_{\tSURE}])\ge 1-\gamma$.
    which follows by another application of Markov's inequality.

\end{proof}

\bibliographystyle{amsalpha}
\bibliography{2nd-order-stein.bib}

\end{document}